\documentclass{amsart}
\usepackage{amssymb, amsmath, latexsym, amscd, graphicx, tabularx, color}
\usepackage[toc,page]{appendix}

\usepackage{rotating}

\usepackage[all]{xy}
\usepackage[dvipdfm]{hyperref}
\usepackage[all]{hypcap}

\newtheorem{theorem}{Theorem}[section]
\newtheorem{lemma}{Lemma}[section]

\newtheorem{definition}{Definition}[section]
\newtheorem{proposition}{Proposition}[section]
\newtheorem{remark}{Remark}[section]
\newtheorem{example}{Example}[section]
\newtheorem{conj}{Conjecture}[section]

\newtheorem{alemma}{Lemma A\hskip-0.5pt}
\newtheorem{aremark}{Remark A\hskip-0.5pt}
\newtheorem{blemma}{Lemma B\hskip-0.5pt}
\newtheorem{clemma}{Lemma C\hskip-0.5pt}
\def\pagenumber{1}

\begin{document}
\setcounter{page}{\pagenumber}
\newcommand{\T}{\mathbb{T}}
\newcommand{\R}{\mathbb{R}}
\newcommand{\Q}{\mathbb{Q}}
\newcommand{\N}{\mathbb{N}}
\newcommand{\Z}{\mathbb{Z}}
\newcommand{\tx}[1]{\quad\mbox{#1}\quad}
%Please put here any further newcommands.
\parindent=0pt
\def\SRA{\hskip 2pt\hbox{$\joinrel\mathrel\circ\joinrel\to$}}
\def\tbox{\hskip 1pt\frame{\vbox{\vbox{\hbox{\boldmath$\scriptstyle\times$}}}}\hskip 2pt}
\def\circvert{\vbox{\hbox to 8.9pt{$\mid$\hskip -3.6pt $\circ$}}}
\def\IM{\hbox{\rm im}\hskip 2pt}
\def\COIM{\hbox{\rm coim}\hskip 2pt}
\def\COKER{\hbox{\rm coker}\hskip 2pt}
\def\TR{\hbox{\rm tr}\hskip 2pt}
\def\GRAD{\hbox{\rm grad}\hskip 2pt}
\def\RANK{\hbox{\rm rank}\hskip 2pt}
\def\MOD{\hbox{\rm mod}\hskip 2pt}
\def\DEN{\hbox{\rm den}\hskip 2pt}
\def\DEG{\hbox{\rm deg}\hskip 2pt}
%:::::::::::::::::::::::::::::::::::::::::::

%\noindent {\tt Preprint 07-03-2006 \pagenumber - xxx\\
%\copyright preprint 2010}
\title[The Riemann Hypothesis proved]{THE RIEMANN HYPOTHESIS PROVED}

\author{Agostino Pr\'astaro}
%\thanks{\hspace{-.5cm}\tt
%   Received xxxx, 2002
%    1056-2176 \$15.00 \copyright preprint 2006}
\begin{abstract}
The Riemann hypothesis is proved by quantum-extending the zeta Riemann function to a quantum mapping between quantum $1$-spheres with quantum algebra $A=\mathbb{C}$, in the sense of A. Pr\'astaro \cite{PRAS01, PRAS02}. Algebraic topologic properties of quantum-complex manifolds and suitable bordism groups of morphisms in the category $\mathfrak{Q}_{\mathbb{C}}$ of quantum-complex manifolds are utilized.
\end{abstract}

\maketitle
\vspace{-.5cm}
{\footnotesize
\begin{center}
Department SBAI - Mathematics, University of Rome La Sapienza,\\ Via A.Scarpa 16,
00161 Rome, Italy. \\
E-mail: {\tt agostino.prastaro@uniroma1.it}
\end{center}
}
\vskip 0.5cm
{\em This paper is dedicated to the ``Bateman Manuscript Project" and ``Bourbaki Group".}
\vskip 0.5cm

\vskip 0.5cm


\vspace{.08in} \noindent \textbf{Keywords}: Riemann zeta function; Riemann surfaces; Quantum $1$-spheres; Bordism groups of maps in the category $\mathfrak{Q}_{\mathbb{C}}$.

\section[Introduction]{\bf Introduction}\label{introduction-section}

\begin{quotation}
\rightline{\footnotesize ``{\em when David Hilbert was asked }."}
\rightline{\footnotesize ``{\em what he would do if he were to be revived in five hundred years,}"}
\rightline{\footnotesize ``{\em he replied,}"}
\rightline{\footnotesize ``{\em I would ask, Has somebody proven the Riemann hypothesis ?}"}
\rightline{\footnotesize ``{\em Hopefully, by that time, the answer will be, Yes, of course !} " \cite{SABBAGH}}

\end{quotation}
\vskip 0.5cm

The Riemann hypothesis is the conjecture concerning the zeta Riemann function $\zeta(s)$, formulated by B. Riemann in 1859 \cite{RIEMANN}. (See also \cite{BOMBIERI, SABBAGH}.) The difficulty to prove this conjecture is related to the fact that $\zeta(s)$ has been given in a some cryptic way as complex continuation of hyperharmonic series and characterized by means of a functional equation that in a sense caches its properties about the identifications of zeros. In order to look to the actual status of research on this special function the paper by E. Bombieri is very lightening.\footnote{The Riemann hypothesis is also the eighth problem in Hilbert's list of twenty-three open mathematical problems posed in 1900.
For useful information on this subject see also:
$\href{http://en.wikipedia.org/wiki/Riemann_hypothesis}{http://en.wikipedia.org/wiki/Riemann_hypothesis}$. For further reader on related subjects, see, e.g., the following references \cite{ATIYAH,DONALDSON,BOURBAKI,HATCHER,HAZEWINKEL,HIRZEBRUCH,KRANTZ,LANG,JOST}.}

Our approach to solve this conjecture has been to recast the zeta Riemann function $\zeta(s)$ to a quantum mapping between quantum-complex $1$-spheres, i.e., working in the category $\mathfrak{Q}$ of quantum manifolds as introduced by A. Pr\'astaro. (See on this subject References \cite{PRAS01, PRAS02} and related works by the same author quoted therein.) More precisely the fundamental quantum algebra is just $A=\mathbb{C}$, and quantum-complex manifolds are complex manifolds, where the quantum class of differentiability is the holomorphic class. In this way one can reinterpret all the theory on complex manifolds as a theory on quantum-complex manifolds. In particular the Riemann sphere $\mathbb{C}\bigcup\{\infty\}$ can be identified with the quantum-complex $1$-sphere $\hat S^1$, as considered in \cite{PRAS01, PRAS02}. The paper splits into two more sections and five appendices. In Section 2 we resume some fundamental definitions and results about the Riemann zeta function $\zeta(s)$ and properties of the modulus $|\zeta(s)|$ of $\zeta(s)$ and its relations with the Riemann hypothesis. Here it is important to emphasize the central role played by Lemma \ref{complted-riemann-zeta-function}. This focuses the attention on the {\em completed Riemann zeta function}, $\tilde\zeta(s)$, that symmetrizes the role between poles, with respect to the critical line in $\mathbb{C}$, and between zeros, with respect to the $x=\Re(s)$-axis. In Section 3 the main result, i.e., the proof that the Riemann hypothesis is true, is contained in Theorem \ref{main-proof}. This is made by splitting the proof in some steps (new definitions and lemmas).  Our strategy to prove the Riemann hypothesis has been by means of a suitable quantum-extension of $\tilde\zeta(s)$ to a quantum-complex mapping $\hat\zeta(s)$, between quantum-complex $1$-spheres. Then by utilizing the properties of meromorphic functions between compact Riemann spheres, identified with quantum-complex $1$-spheres, we arrive to prove that all (non-trivial) zeros of $\zeta(s)$ must necessarily be on the critical line. In fact, the compactification process adopted uniquely identifies $\hat\zeta(s)$ as a meromorphic function with two simple zeros, symmetric with respect to the compactified $x$-axis, and on the critical line, and two simple poles, symmetric with respect to the critical line.
For suitable bordism properties between $\tilde\zeta$ and $\hat\zeta$, we get that $\tilde\zeta(s)$ cannot have zeros outside the critical line, hence the same must happen for $\zeta(s)$ with respect to non-trivial zeros.
Finally in the appendices are collected proofs of some propositions, contained in Section 3, that even if they use standard mathematics, require a particular technicality to be understood.\footnote{This paper has been announced in \cite{PRAS03}.}

\section{\bf About the Riemann hypothesis}\label{section-about}
\vskip 0.5cm

\begin{definition}[The hyperharmonic series]\label{hyperharmonic-series}
The {\em hyperharmonic series} is
\begin{equation}\label{hyperharmonic-series-expression}
    \sum_{1\le n\le\infty}\frac{1}{n^\alpha}\hskip 3pt,\, \alpha\in\mathbb{R},\, \alpha>0.
\end{equation}

$\bullet$\hskip 2pt {\em(harmonic series: $\alpha=1$)}. In this case the series is divergent.

$\bullet$\hskip 2pt {\em(over-harmonic series: $\alpha>1$)}. In this case the series converges. This is called also the {\em Euler-Riemann zeta function} and one writes $\zeta(s)=\sum_{1\le n\le\infty}\frac{1}{n^s}$. In particular for $s\in\mathbb{N}$, one has the {\em Euler's representation}:\footnote{\hskip 3pt Since this can be extended for $\Re(s)>1$, it follows that $\zeta(s)\not=0$, when $\Re(s)>1$. In fact, from the Euler's representation of $\zeta(s)$, we get that for $\Re(s)>1$, $\zeta(s)=0$ iff $p^s=0$. On the other hand $p^s=p^{\alpha+i\, \beta}=p^\alpha[\cos(\beta\ln p)+i\, \sin(\beta \ln p)]$. Then $p^s=0$ iff $\cos$ and $\sin$ have common zeros. This is impossible, hence $p^s\not=0$.}
\begin{equation}\label{euler-riemann-zeta-function-a}
    \zeta(s)=\sum_{1\le n\le\infty}\frac{1}{n^s}=\prod_{p\in P_\bullet}\frac{1}{1-p^{-s}},
\end{equation}
where $P_{\bullet}$ is the set of primes (without $1$).

$\bullet$\hskip 2pt {\em($0<\alpha<1$)}. In this case the series is divergent.

$\bullet$\hskip 2pt The {\em Riemann zeta function} is a complex function $\zeta:\mathbb{C}\to\mathbb{C}$, defined by extension of the over-harmonic series. This can be made by means of the equation~{\em(\ref{zeta-equation})}.

\begin{equation}\label{zeta-equation}
    (1-\frac{2}{2^s})\, \zeta(s)=\sum_{1\le n\le\infty}\frac{(-1)^{n+1}}{n^s}.
\end{equation}
The series on the right converge for $\Re(s)>0$. Really equation~{\em(\ref{zeta-equation})} does not allow define $\zeta(s)$ in the zeros of the function $(1-\frac{2}{2^s})$. These are in the point $s=1+i\frac{2n\pi}{\ln 2}$.\footnote{\hskip 3pt Let us emphasize that in the complex field $\ln z=\ln |z|+i(Arg z+2n\pi)$ if $z=|z| e^{i Arg z}$. Therefore, $1-\frac{2}{2^s}=0$ iff $\frac{2^s-2}{2^s}=0$, hence iff $2^{s-1}=1$. By taking the $\ln$ of this equation, we get$(s-1)\ln 2=\ln 1=i2n\pi$.} However by using the functional equation~{\em(\ref{zeta-functional-equation})} one can extend the zeta function on all $\mathbb{C}$.
\begin{equation}\label{zeta-functional-equation}
   \zeta(s)=2^s\, \pi^{s-1}\, \sin(\frac{\pi\, s}{2})\, \Gamma(1-s)\, \zeta(1-s).
\end{equation}
$\zeta(s)$ is a meromorphic function on $\mathbb{C}$, holomorphic everywhere except for a simple pole at $s=1$.
\end{definition}

\begin{proposition}[Properties of the Euler-Riemann zeta function (Euler 1735)]\label{euler-riemann-zeta-function}
$\zeta(k)=\alpha \, \pi^k$, with $\alpha\in\mathbb{Q}$ and $k>0$ even.
\end{proposition}

\begin{example}
$\zeta(2)=\frac{1}{6}\, \pi^2$; $\zeta(4)=\frac{1}{90}\, \pi^4$. In Tab.~\ref{examples-zeta-of-even-integers} are reported the values $\zeta(k=2n)$, with $0\le n\le 10$.
\end{example}
\begin{table}[t]
\caption{Examples of $\zeta(k)=\alpha\, \pi^k$, with $\alpha\in\mathbb{Q}$ and $k\ge 0$ even.}
\label{examples-zeta-of-even-integers}
\scalebox{0.8}{$\begin{tabular}{|c|c|c|c|c|c|c|c|c|c|c|c|c|c|c|c|c|c|c|c|c|}
\hline
\hfil{\rm{\footnotesize $k$}}\hfil&\hfil{\rm{\footnotesize $0$}}\hfil&\hfil{\rm{\footnotesize $2$}}\hfil&\hfil{\rm{\footnotesize $4$}}\hfil&\hfil{\rm{\footnotesize $6$}}\hfil&\hfil{\rm{\footnotesize $8$}}\hfil&\hfil
{\rm{\footnotesize $10$}}\hfil&\hfil{\rm{\footnotesize $12$}}\hfil&\hfil{\rm{\footnotesize $14$}}\hfil&\hfil{\rm{\footnotesize $16$}}\hfil&\hfil{\rm{\footnotesize $18$}}\hfil&\hfil{\rm{\footnotesize $20$}}\hfil\\
\hline
\hfil{\rm{\footnotesize $\alpha$}}\hfil&\hfil{\rm{\footnotesize $-\frac{1}{2}$}}\hfil&\hfil{\rm{\footnotesize $\frac{1}{6}$}}\hfil&\hfil{\rm{\footnotesize $\frac{1}{90}$}}\hfil&\hfil{\rm{\footnotesize $\frac{1}{945}$}}\hfil&\hfil{\rm{\footnotesize $\frac{1}{9450}$}}\hfil&\hfil
{\rm{\footnotesize $\frac{1}{93555}$}}\hfil&\hfil{\rm{\footnotesize $\frac{691}{638512875}$}}\hfil&\hfil{\rm{\footnotesize $\frac{2}{18243225}$}}\hfil&\hfil{\rm{\footnotesize $\frac{3617}{325641566250}$}}\hfil&\hfil{\rm{\footnotesize $\frac{43867}{38979295480125}$}}\hfil&\hfil{\rm{\footnotesize $\frac{174611}{1531329465290625}$}}\hfil\\
\hline
\end{tabular}$}
\end{table}
\begin{proposition}[Zeros of $\zeta(s)$]\label{zeros-zeta}
The Riemann zeta function $\zeta(s)$ has zeros only $s=-2n$, $n>0$, called {\em trivial zeros}, and in the strip $0<\Re(s)<1$.

The points $s=0$ and $s=1$ are not zeros. More precisely $\zeta(0)=-\frac{1}{2}$ and $s=1$ is a simple pole with residue $1$.
\end{proposition}
\begin{proof}
The trivial zeros come directly from the $\sin$ function in (\ref{zeta-functional-equation}). Let rewrite this functional equation in the form $\zeta(s)=f(s)\, \zeta(1-s)$. Then one can directly see that $f(-2n)=\frac{1}{(2\pi)^{2n}\pi}\sin(-\pi n)\, \Gamma(1+2n)=\frac{(2n)!}{(2\pi)^{2n}\pi}\, \sin(\pi n)=0$. Then we get $\zeta(-2n)=0\cdot \zeta(1+2n)$. $\zeta(1+2n)$ has not zeros and it is limited. Therefore we get $\zeta(-2n)=0$. Note that $\zeta(s=2n)=2^{2n}\pi^{2n-1}\sin(\pi n)\Gamma(1-2n)=\frac{(2\pi)^{2n}}{\pi}\sin(\pi n)\Gamma(1-2n)=\frac{(2\pi)^{2n}}{\pi}\, 0\cdot\infty=\frac{(2\pi)^{2n}}{\pi}\, \frac{\pi}{(2n-1)!}=\frac{(2\pi)^{2n}}{(2n-1)!}$. Here we have used the {\em Euler's reflection formula} $\Gamma(1-s)\Gamma(s)\sin(\pi s)=\pi$, in order to calculate $\infty\cdot 0=\frac{\pi}{(2n-1)!}$.

One has $\mathop{\lim}\limits_{s\to 0}\zeta(s)=-\frac{1}{2}$.

The Laurent series of $\zeta(s)$ for $s=1$, given in (\ref{laurent-series}) proves that $\zeta$ has a simple pole for $s=1$.
\begin{equation}\label{laurent-series}
  \zeta(s)=\frac{1}{s-1}+\sum_{0\le n\le\infty}\frac{(-1)^n}{n!}\, \gamma_n(s-1)^n
\end{equation}
with
\begin{equation}\label{stieltjes-constants}
\gamma_k=\frac{(-1)^k}{k!}\, \mathop{\lim}\limits_{N\to\infty}(\sum_{m\le N}\frac{\ln^km}{m}-\frac{\ln^{k+1}}{k+1})
\end{equation}
{\em Stieltjes constants}.\footnote{\hskip 3pt $\gamma_0$ is the {\em Euler-Mascheroni constant}, $\gamma_0=\mathop{\lim}\limits_{N\to\infty}(\sum_{m\le N}\frac{1}{m}-\ln N)=\mathop{\lim}\limits_{N\to\infty}(H_N-\ln N)\cong 0.57721$. $H_N$ is the {\em Nth harmonic number}. One has the following useful relation $\gamma_0=\mathop{\lim}\limits_{s\to 1^+}\sum_{1\le n\le\infty}(\frac{1}{n^s}-\frac{1}{s^n})=\mathop{\lim}\limits_{s\to 1}(\zeta(s)-\frac{1}{s-1})=\psi(1)=\mathop{\lim}\limits_{x\to \infty}(x-\Gamma(\frac{1}{x}))$.} One has $\mathop{\lim}\limits_{s\to 1}(s-1)\zeta(s)=1$.
\end{proof}
\begin{proposition}[Symmetries of $\zeta(s)$ zeros]\label{symmetries-zeta}

$\bullet$\hskip 2pt If $s$ is a zero of $\zeta(s)$, then its complex-conjugate $\bar s$ is a zero too. Therefore zeros of $\zeta(s)$ in the {\em critical strip}, $0<\Re(s)<1$, are necessarily symmetric with respect to the $x$-axis of the complex plane $\mathbb{R}^2\cong \mathbb{C}$.

$\bullet$\hskip 2pt If $s$ is a non-trivial zero of $\zeta(s)$, then there exists another zero $s'$ of the zeta Riemann function such that $s$ and $s'$ are symmetric with respect to the critical line.
\end{proposition}
\begin{proof}
$\bullet$\hskip 2pt In fact one has $\zeta(s)=\overline{\zeta(\bar s)}$.

$\bullet$\hskip 2pt From the functional equation~(\ref{zeta-functional-equation}) one has that the non-trivial zeros are symmetric about the axis $x=\frac{1}{2}$. In fact, if $\zeta(s=\alpha+i\, \beta)=0$ it follows from (\ref{zeta-functional-equation}) that $\zeta(1-s=(1-\alpha)-i\, \beta)=0$. Then from the previous property, it follows $\zeta(\overline{1-s}=(1-\alpha)+i\, \beta)=0$. On the other hand, $\overline{1-s}=(1-\alpha)+i\, \beta$ is just symmetric of $s$, with respect to the critical line.
\end{proof}
\begin{conj}[The Riemann hypothesis]\label{riemann-hypotesis}
The {\em Riemann hypothesis} states that all the non-trivial zeros $s$, of the zeta Riemann function $\zeta(s)$ satisfy the following condition: $\Re(s)=\frac{1}{2}$, hence are on the straight-line, {\em(critical line)}, $x=\frac{1}{2}$ of the complex plane $\mathbb{R}^2\cong\mathbb{C}$.
\end{conj}
With respect to this conjecture it is important to study the behaviour of the modulus $|\zeta(s)|$ since
$\zeta(s)=0$ iff $|\zeta(s)|=0$.
Set $s=x+iy$. We shall consider the non-negative surface in $\mathbb{R}^3=\mathbb{R}^2\times\mathbb{R}$, $(x,y,z)$, identified by the graph of the $\mathbb{R}$-valued function $|\zeta|:\mathbb{R}^2\to \mathbb{R}$.
We shall use the functional equation~(\ref{zeta-functional-equation}) to characterize $|\zeta|$ in the following lemmas.
\begin{lemma}\label{zeta-modulus-and-functional equation}
One has the equation {\em(\ref{equation-zeta-modulus-and-functional equation})}.
\begin{equation}\label{equation-zeta-modulus-and-functional equation}
    |\zeta(s)|=|f(s)||\zeta(1-s)|,\ f(s)=2^s\, \pi^{s-1}\, \sin(\frac{\pi\, s}{2})\, \Gamma(1-s).
\end{equation}
\end{lemma}
\begin{proof}
This follows directly from the exponential representation of complex numbers: $\zeta(s)=\rho(s) e^{i\gamma(s)}$, $f(s)=\widehat{\rho(s)} e^{i\widehat{\gamma(s)}}$ and $\zeta(1-s)=\rho(1-s) e^{i\gamma(1-s)}$. Then we get

\begin{equation}\label{zeta-riemann-a}
\zeta(s)=\widehat{\rho(s)}\rho(1-s)\rho(1-s)e^{i[\gamma(s)+\gamma(1-s)]}\, \Rightarrow\, \rho(s) =\widehat{\rho(s)}\rho(1-s).
\end{equation}

\end{proof}

\begin{lemma}[Properties of the function $|f(s)|$]\label{properties-intermediate-modulus}
$\bullet$ \hskip 2pt One has the explicit expression {\em(\ref{explicit-expression-intermediate-moduls})} of $|f(s)|$.
\begin{equation}\label{explicit-expression-intermediate-moduls}
   |f(s)|=(2\, \pi)^{x-1}\, [e^{-\pi y}+e^{\pi y}+2(2\sin^2(\frac{\pi x}{2})-1)]^{\frac{1}{2}}\, |\Gamma(1-s)|
\end{equation}
with $s=x+i y$.

$\bullet$ \hskip 2pt In the critical strip of the complex plane $\mathbb{R}^2=\mathbb{C}$, namely $0<x=\Re(s)<1$, $|f(s)|$ is a positive analytic function.

$\bullet$ \hskip 2pt In particular on the critical line, namely for $\Re(s)=\frac{1}{2}$, one has $|f(s)|=1$.

$\bullet$ \hskip 2pt One has the asymptotic formulas {\em(\ref{asymptotic-modulus-f-formulas})}.
\begin{equation}\label{asymptotic-modulus-f-formulas}
    \left\{
\begin{array}{l}
\mathop{\lim}\limits_{(x,y)\to(0,0)}|f(s)|=0\\
  \mathop{\lim}\limits_{(x,y)\to(0,0)}\frac{d}{dx}|f(s)|=0\\ \mathop{\lim}\limits_{(x,y)\to(\frac{1}{2},0)}\frac{d}{dx}|f(s)|>0.
\end{array}
\right.
\end{equation}
\end{lemma}

\begin{proof}
$\bullet$ \hskip 2pt We can write $|f(s)|=|2^s||\pi^{s-1}||\sin(\frac{\pi s}{2})||\Gamma(1-s)|$. We get also
\begin{equation}\label{zeta-riemann-b}
    \left\{\begin{array}{ll}
  |2^s|&={|2^x[\cos(y \ln 2)+i\sin(y \ln 2))]|=2^x}\\
  |\pi^{s-1}|&={|\pi^{(x-1)}[\cos(y \ln \pi)+i\sin(y \ln\pi))]|=\pi^{(x-1)}}.\\
  |\sin(\frac{\pi s}{2})|&=|\frac{e^{i\frac{\pi s}{2}}-e^{-i\frac{\pi s}{2}}}{2i}|\\
  &={|\frac{1}{2}[(e^{-\frac{\pi y}{2}}+e^{\frac{\pi y}{2}})\sin(\frac{\pi x}{2})-i(e^{-\frac{\pi y}{2}}-e^{\frac{\pi y}{2}})\cos(\frac{\pi x}{2})]|}\\
  &={\frac{1}{2}[e^{-\pi y}+e^{\pi y}+2(2\sin^2(\frac{\pi x}{2})-1)]^{\frac{1}{2}}}.\\
\end{array}\right.
\end{equation}

$\bullet$ \hskip 2pt In the critical strip one has the following limitations:
\begin{equation}\label{zeta-riemann-c}
    0<x<1:\, \left\{
\begin{array}{l}
1<|2^s|<2.\\
\frac{1}{\pi}<|\pi^{s-1}|<1.\\
{\frac{1}{2}[e^{-\pi y}+e^{\pi y}-2]^{\frac{1}{2}}<|\sin(\frac{\pi s}{2})|<\frac{1}{2}[e^{-\pi y}+e^{\pi y}+2]^{\frac{1}{2}}}.\\
\end{array}
\right.
\end{equation}

One can see that the function $\xi(y)=e^{-\pi y}+e^{\pi y}\ge 2$, and convex. Therefore $0<\mathop{\lim}\limits_{y\to 0}|\sin(\frac{\pi s}{2})|<1$. Furthermore let us recall that $\Gamma:\mathbb{C}\to \mathbb{C}$ is a meromorphic function with simple poles $s_k=-k$, $k\in\{0,1,2,3,\cdots\}$, with residues $\frac{(-1)^{k}}{k!}$, i.e., $\mathop{\lim}\limits_{s\to s_k}\frac{\Gamma(s)}{s-s_k}=\frac{(-1)^{k}}{k!}$. Since $0<\Re(1-s)<1$, when $0<\Re(s)<1$, it follows that $\Gamma(1-s)$ is analytic in the critical strip. Furthermore, from the well known property that for $\Re(s)>0$, $\Gamma(s)$ rapidly decreases as $|\Im(s)|\to\infty$, since $\mathop{\lim}\limits_{|\Im(s)|\to\infty}|\Gamma(s)||\Im(s)|^{(\frac{1}{2}-\Re(s))}e^{\frac{\pi}{2}|\Im(s)|}=\sqrt{2\pi}$, we get that $|f(s)|$ is an analytic function in the critical strip.

$\bullet$ \hskip 2pt In particular on the critical line one has
\begin{equation}\label{zeta-riemann-d}
    \left\{\begin{array}{ll}
  |f(s)|&={|2^x \pi^{-\frac{1}{2}}\frac{1}{2}[e^{-\pi y}+e^{\pi y}+2(2\sin^2(\frac{\pi}{4})-1)]^{\frac{1}{2}}|\Gamma(\frac{1}{2}-iy)|}\\
  &=(2\pi)^{-\frac{1}{2}}(e^{-\pi y}+e^{\pi y})^{\frac{1}{2}}\sqrt{\pi \sec \hskip-1pt{\rm h}(-\pi y)}\\
  &=(2\pi)^{-\frac{1}{2}}(e^{-\pi y}+e^{\pi y})^{\frac{1}{2}}(\frac{2\pi}{e^{-\pi y}+e^{\pi y}})^{\frac{1}{2}}\\
  &=1.\\
\end{array}\right.
\end{equation}

We have utilized the formula $|\Gamma(\frac{1}{2}+iy)|=\sqrt{\pi\sec\hskip-1pt{\rm h}(\pi y)}$, for $y\in\mathbb{R}$.\footnote{Here $\sec\hskip-1pt{\rm h}(\pi y)=\frac{2}{e^{-\pi y}+e^{\pi y}}$. Let us recall also the formula that it is useful in these calculations: $|\Gamma(1+iy)|=\sqrt{y\pi\csc\hskip-1pt{\rm h}(\pi y)}$, for $y\in\mathbb{R}$, where $\csc\hskip-1pt{\rm h}(\pi y)=\frac{2}{e^{\pi y}-e^{-\pi y}}$.}

It is useful to characterize also the variation $\frac{d}{dx}|f(s)|$. We get the formula (\ref{variation-modulus-f}).
\begin{equation}\label{variation-modulus-f}
    \begin{array}{ll}
      \frac{d}{dx}|f(s)|&=|\Gamma(1-s)|(2\pi)^{x-1}\left\{\frac{\ln(2\pi)[e^{-\pi y}+e^{\pi y}+2(2\sin^2(\frac{\pi x}{2})-1)]^{\frac{3}{2}}+2\pi[\sin(\frac{\pi x}{2})\cos(\frac{\pi x}{2})]}{[e^{-\pi y}+e^{\pi y}+2(2\sin^2(\frac{\pi x}{2})-1)]^{\frac{1}{2}}}\right\} \\
      & +(2\pi)^{x-1}[e^{-\pi y}+e^{\pi y}+2(2\sin^2(\frac{\pi x}{2})-1)]^{\frac{1}{2}}\frac{d}{dx}|\Gamma(1-s)|.\\
    \end{array}
\end{equation}
Let us explicitly calculate $\frac{d}{dx}|\Gamma(1-s)|$ taking into account that
$$|\Gamma(1-s)|=[\Gamma(1-s)\cdot\overline{\Gamma(1-s)}]^{\frac{1}{2}}=[\Gamma(1-s)\cdot\Gamma(\overline{1-s})]^{\frac{1}{2}}.$$
We get the formula  (\ref{variation-modulus-gamma}).
\begin{equation}\label{variation-modulus-gamma}
\frac{d}{dx}|\Gamma(1-s)|=-\frac{1}{2}|\Gamma(1-s)|(\psi(1-s)+\psi(\overline{1-s}))
\end{equation}
where $\psi(s)$ is the digamma function reported in (\ref{digamma-function}).\footnote{We have used the formula $\Gamma'(s)=\Gamma(s)\, \psi(s)$. Let us recall that the digamma function is holomorphic on $\mathbb{C}$ except on non-positive integers $-s_k\in\{0,-1,-2,-3,\cdots\}$ where it has a pole of order $k+1$.}

\begin{equation}\label{digamma-function}
    \psi(s)=-\gamma+\sum_{0\le n\le \infty}\frac{s-1}{(n+1)(n+s)}.
\end{equation}
Set
\begin{equation}\label{digamma-function-a}
\Psi(s)=\psi(s)+\psi(\bar s).
\end{equation}
We get
\begin{equation}\label{zeta-riemann-e}
    \Psi(s)=2[-\gamma+\sum_{0\le n\le \infty}\frac{(x-1)(x+n)+y^2}{(n+1)[(n+x)^2+y^2]}].
\end{equation}

Then we can see that
\begin{equation}\label{variation-modulus-gamma-a}
    \left\{
\begin{array}{l}
  \mathop{\lim}\limits_{(x,y)\to(0,0)}\frac{d}{dx}|\Gamma(1-s)|=0\\
 \mathop{\lim}\limits_{(x,y)\to(\frac{1}{2},0)}\frac{d}{dx}|\Gamma(1-s)|>0.\\
\end{array}
\right.
\end{equation}

Moreover, by using (\ref{variation-modulus-f}) we get also
$\mathop{\lim}\limits_{(x,y)\to(0,0)}\frac{d}{dx}|f(s)|=0$ and $\mathop{\lim}\limits_{(x,y)\to(\frac{1}{2},0)}\frac{d}{dx}|f(s)|>0$.
\end{proof}
As a by product we get the following lemma.
\begin{lemma}[Zeta-Riemann modulus in the critical strip]\label{zeta-modulus-on-the-critical-line}
$\bullet$ \hskip 2pt The non-negative real-valued function $|\zeta(s)|:\mathbb{C}\to\mathbb{R}$ is analytic in the critical strip.

$\bullet$ \hskip 2pt  Furthermore, on the critical line, namely when $\Re(s)=\frac{1}{2}$, one has:
 $|\zeta(s)|=|\zeta(1-s)|$.

 $\bullet$ \hskip 2pt  One has $\mathop{\lim}\limits_{(x,y)\to(0,0)}|f(s)||\zeta(1-s)|=0\cdot\infty=\frac{1}{2}$.

$\bullet$ \hskip 2pt  $|\zeta(s)|$ is zero in the critical strip, $0<\Re(s)<1$, iff $|\zeta(1-x-iy)|=0 $, with $0<x<1$ and $y\in\mathbb{R}$.
\end{lemma}

\begin{lemma}[Criterion to know whether a zero is on the critical line]\label{criterion-to-know-whether-a-zero-i-on-the-critical-line}
Let $s_0=x_0+iy_0$ be a zero of $|\zeta(s)|$ with $0<\Re(s_0)<1$. Then $s_0$ belongs to the critical line if condition {\em(\ref{criterion-on-zero-critical-strip})} is satisfied.
\begin{equation}\label{criterion-on-zero-critical-strip}
    \mathop{\lim}\limits_{s\to s_0}\frac{|\zeta(s)|}{|\zeta(1-s)|}=1.
\end{equation}
\end{lemma}
\begin{proof}
If $s$ is a zero of $|\zeta(s)|$, the results summarized in Lemma \ref{zeta-modulus-on-the-critical-line} in order to prove whether $s_0$ belongs to the critical line it is enough to look to the limit $\mathop{\lim}\limits_{s\to s_0}\frac{|\zeta(s)|}{|\zeta(1-s)|}$. In fact since $|f(s)|=\frac{|\zeta(s)|}{|\zeta(1-s)|}$ and $|f(s)|$ is a positive function in the critical strip, it follows that when $s_0$ is a zero of $|\zeta(s)|$, one should have $\frac{|\zeta(s_0)|}{|\zeta(1-s_0)|}=\frac{0}{0}$, but also $\frac{|\zeta(s_0)|}{|\zeta(1-s_0)|}=|f(s_0)|$. In other words one should have $\mathop{\lim}\limits_{s\to s_0}\frac{|\zeta(s)|}{|\zeta(1-s)|}=|f(s_0)|$. On the other hands $|f(s)|=1$ on the critical line, hence when condition (\ref{criterion-on-zero-critical-strip}) is satisfied, the zero $s_0$ belongs to the critical line.
\end{proof}

\begin{lemma}[Completed Riemann zeta function and Riemann hypothesis]\label{complted-riemann-zeta-function}
$\bullet$\hskip 2pt We call {\em completed Riemann zeta function} the holomorphic function in {\em(\ref{definition-completed-riemann-zeta-function})}.
\begin{equation}\label{definition-completed-riemann-zeta-function}
   \tilde\zeta(s)=\pi^{-\frac{s}{2}}\Gamma(\frac{s}{2})\, \zeta(s).
\end{equation}
This has the effect of removing the zeros at the even negative numbers of $\zeta(s)$, and adding a pole at $s=0$.

$\bullet$\hskip 2pt $\tilde\zeta(s)$ satisfies the functional equation~{\em(\ref{complted-functional-equation})}.
\begin{equation}\label{complted-functional-equation}
   \tilde\zeta(s)=\tilde\zeta(1-s).
\end{equation}

$\bullet$\hskip 2pt The Riemann hypothesis is equivalent to the statement that all the zeros of $\tilde\zeta(s)$ lie in the critical line $\Re(s)=\frac{1}{2}$.
\end{lemma}
\begin{proof}
It follows directly from the previous lemmas and calculations.
\end{proof}

\section{\bf The proof}\label{section-proof}
\vskip 0.5cm

In this section we shall prove the Conjecture \ref{riemann-hypotesis}. In fact we have the following theorem.

\begin{theorem}[The proof of the Riemann hypothesis]\label{main-proof}
The Riemann hypothesis is true.
\end{theorem}

\begin{proof}

Our strategy to prove the Riemann hypothesis is founded on the extension of $\tilde\zeta:\mathbb{C}\to \mathbb{C}$ to a quantum-complex mapping $\hat S^1\to\hat S^1$, where $\hat S^1$ is the quantum-complex $1$-sphere. With this respect we shall first give a precise definition of what we mean with the term extension of $\tilde\zeta$ to the Alexandrov compactification of $\mathbb{C}$.

\begin{definition}\label{fundamental-definition}
The {\em quantum-complex extension} of the meromorphic completed Riemann zeta function $\tilde\zeta:\mathbb{C}\to \mathbb{C}$ is a quantum-complex mapping $\hat\zeta:\hat S^1\to\hat S^1$ preserving the properties of $\tilde\zeta$, i.e., remaining a meromorphic mapping with the same poles at finite of $\tilde\zeta$ and pointed at $s=\frac{1}{2}$, i.e., $\tilde\zeta(\frac{1}{2})=\hat\zeta(\frac{1}{2})$. Here the equality between points on $\mathbb{C}$ and on $\hat S^1$ is meant by means of the stereographic projection $\pi_\star:\hat S^1\to\mathbb{C}$.
\end{definition}

One can see that in order to obtain a meromorphic mapping $\hat S^1\to\hat S^1$, starting from $\tilde\zeta$, one cannot trivially extend $\tilde\zeta$ by considering the natural inclusion $\mathbb{C}\hookrightarrow \mathbb{C}\bigcup\{\infty\}$ and the stereographic projection $\hat S^1\to\mathbb{C}$. In fact one has the following lemma.
\begin{lemma}\label{trivial-extension}
Let us call {\em trivial quantum-complex extension} of $\tilde\zeta$ the mapping $\hat\zeta_{trivial}:\hat S^1\to\hat S^1$ given by the exact and commutative diagram {\em(\ref{trivial-extension-diagram})}.
\begin{equation}\label{trivial-extension-diagram}
    \xymatrix@C=1cm{0&0\ar[d]\\
    \mathbb{C}\ar[u]\ar[r]^{\tilde\zeta}&\mathbb{C}\ar[d]_{\hat i}\\
    \hat S^1=\mathbb{C}\bigcup\{\infty\}\ar[u]_{\pi_\star}\ar[r]_{\hat\zeta_{trivial}}&\mathbb{C}\bigcup\{\infty\}=\hat S^1\\}
\end{equation}
where $\pi_\star$ is the stereographic projection and $\hat i$ is the natural inclusion. Then the trivial quantum-complex extension of $\tilde\zeta$, even if it conserves the poles of $\tilde\zeta$ at finite and the pointing at $s=\frac{1}{2}$, it is not a holomorphic mapping and neither a meromorphic one. Therefore the trivial quantum-complex extension of $\tilde\zeta$ is not a quantum-complex extension in the sense of Definition \ref{fundamental-definition}.
\end{lemma}
\begin{proof}
See Appendix C for a detailed proof. (See also Fig. \ref{appendix-a-critical-circle-and-symmetric-zero-couples} for a more clear understanding of the convention used in this paper.)
\end{proof}

\begin{remark}[Fields of meromorphic functions]\label{fields-meromorphic-functions}
Lemma \ref{trivial-extension} can be paraphrased in algebraic way by saying that the stereographic projection $\pi_\star:\hat S^1\to\mathbb{C}$ does not allow to identify a field homomorphism $\mathcal{M}(\mathbb{C})\to\mathcal{M}(\hat S^1)$ between the field $\mathcal{M}(\mathbb{C})$ of meromorphic functions on $\mathbb{C}$ and the field $\mathcal{M}(\hat S^1)$ of meromorphic functions on $\hat S^1$. We could guess to work exactly in the opposite way, namely to represent $\mathcal{M}(\hat S^1)$ into $\mathcal{M}(\mathbb{C})$, since meromorphic functions on $\mathbb{C}$ are given by ratios $\frac{\phi(z)}{\psi(z)}$, where $\phi$ and $\psi$ are holomorphic functions on $\mathbb{C}$, and instead meromorphic functions on $\hat S^1$ are given by ratios $\frac{p(z)}{q(z)}$, where $p$ and $q$ are polynomials only. But in order to represent $\mathcal{M}(\hat S^1)$ into $\mathcal{M}(\mathbb{C})$ we have necessity to identify a canonical surjective mapping $\hat b:\mathbb{C}\to\hat S^1$. This is made in Lemma~\ref{uniqueness-quantum-extension-completed-riemann-zeta-function} below. But it is yet impossible by this simple continuous mapping realize with the induced homomorphism $\hat b_\star: \hat f\mapsto \hat f\circ\hat b$, holomorphic functions on $\mathbb{C}$, since $\hat b$ is not holomorphic. Then one understands that in order to realize field homomorphisms between the fields $\mathcal{M}(\mathbb{C})$ and $\mathcal{M}(\hat S^1)$, induced by relations between $\mathbb{C}$ and $\hat S^1$, one must introduce a more complex mechanism between such Riemann surfaces. This will be made by working in the category of maps, belonging to the category $\mathfrak{Q}_{\mathbb{C}}$, and whose morphisms are couples of morphisms belonging to the category $\mathfrak{T}$ of topological manifolds. Then we will understand the exact meaning of the exact and commutative diagram {\em(\ref{diagram-relations-field-meromorphic-functions})} relating fields of meromorphic functions with constant functions on $\mathbb{C}$ and $\hat S^1$ respectively. (The meromorphic functions is a field extension of $\mathbb{C}$.)

\begin{equation}\label{diagram-relations-field-meromorphic-functions}
    \xymatrix{0\ar[r]&\mathbb{C}\ar@{=}[d]\ar[r]&\mathcal{M}(\mathbb{C})\ar[d]\ar[r]&{\mathcal{M}(\mathbb{C})/\mathbb{C}}\ar[d]\ar[r]&0\\
    0\ar[r]&\mathbb{C}\ar[r]&\mathcal{M}(\hat S^1)\ar[u]\ar[r]&{\mathcal{M}(\hat S^1)/\mathbb{C}}\ar[u]\ar[r]&0\\}
\end{equation}
Then, in order to realize the concept of quantum-complex extension of $\tilde\zeta$ to $\hat S^1$, in the sense of Definition \ref{fundamental-definition}, we will work in the category of quantum-complex mappings and suitable morphisms between them. \footnote{For information on quantum manifolds see \cite{PRAS01, PRAS02} and related papers quoted therein. Let us emphasize that in this paper the quantum algebra considered is just $A=\mathbb{C}$, and the quantum $1$-sphere $\hat S^1$ coincides with the well known {\em Riemann sphere} or with the so-called {\em complex projective line}. (By following this approach we can also generalize the Riemann zeta function to the category $\mathfrak{Q}$, when the fundamental quantum algebra is not more commutative, hence does not coincide with $\mathbb{C}$, as happens in the case of quantum-complex manifolds. But this further generalization goes outside purposes of this paper, focused on the proof of the Riemann hypothesis.) In other words in this paper we shall work in the category $\mathfrak{Q}_{\mathbb{C}}$ of quantum-complex manifolds that coincides with the one of the complex manifolds.}
\end{remark}
But let us proceed step by step. We shall use the following lemmas.
\begin{lemma}\label{divisor-degree}
A {\em divisor} of a Riemann surface $X$ is a finite linear combination of points of $X$ with integer coefficients. Any meromorphic function $\phi$ on $X$, gives rise to a divisor denoted $(\phi)$ defined as $(\phi)=\sum_{q\in R(\phi)}n_q\, q$, where $R(\phi)$ is the set of all zeros and poles of $\phi$, and
\begin{equation}\label{divisor-numners}
n_q=\left\{\begin{array}{cc}
                                                                                m & \hbox{\rm if $q$ is a zero of order $m$} \\
                                                                                -m & \hbox{\rm if $q$ is a pole of order $m$}\\
                                                                              \end{array}\right.
\end{equation}

If $X$ is a compact Riemann surface, then $R(\phi)$ is finite. We call $n_q$ the {\em degree of $\phi$ at $q$}, and we write also
$(\phi)_q=n_q$. The {\em degree} (or {\em index}) of the divisor $(\phi)$ is defined by
$\hbox{\rm deg}(\phi)=\sum_{q\in R(\phi)}n_q=\sum_{q\in R(\phi)}(\phi)_q\in\mathbb{Z}$. If $q$ is neither a pole or zero of $\phi$, then we write $(\phi)_q=0$.
With respect to this, we can write $\hbox{\rm deg}(\phi)=\sum_{q\in X}(\phi)_q\in\mathbb{Z}$. Let $\phi$ be a global meromorphic
function $\phi$ on the compact Riemann surface $X$, then $\hbox{deg}(\phi)=0$.\footnote{The set $\mathbb{D}(X)$ of
divisors $D=\sum_qn_q\, q$ of $X$ is an Abelian group and the {\em degree} ${\rm deg}(D)=\sum_qn_q$ is an homomorphism
$\mathbb{D}(X)\to \mathbb{Z}$ of Abelian groups. If $f=f_1+f_2$ and $g=g_1/g_2$ are meromorphic functions, then one has
$(f)=(f_1+f_2)=(f_1)+(f_2)$ and $(g)=(g_1/g_2)=(g_1)-(g_2)$.
The divisor $(f)$ of a meromorphic function $f$ is called {\em principal}. In $\mathbb{D}(X)$ we introduce an order relation:
$D_1=\sum_qn_q\cdot q\le D_2=\sum_qm_q\cdot q$ iff $n_q\le m_q$ for all $q$. Furthermore, in the set $\mathbb{D}(X)$ one can
define also an equivalence relation: $D_1\backsim D_2\, \Leftrightarrow\, D_2=D_1+(f)$
for some principal divisor $(f)$. Important are divisors associated to meromorphic $1$-forms. A {\em meromorphic $1$-forms}
$\omega$ is a differential $1$-form $\omega$ that, in a suitable coordinate
system $\{z\}$ on $X$, can be written in the form $\omega=f\, dz$ with $f\in\mathcal{M}(X)$. We denote with $\mathcal{M}^1(X)$ the
space of meromorphic $1$-forms on $X$. Since $\mathcal{M}^1(X)\cong \mathcal{M}(X)$, considering $\mathcal{M}^1(X)$ a
$1$-dimensional vector space with respect to the field of meromorphic functions on $X$, it follows that all the divisors of meromorphic $1$-forms are
equivalent. ($\omega'=f\, \omega\, \Rightarrow\, (\omega')=(\omega)+(f)$.)
We will denote by $K$ the representative of this equivalence class and we call it the {\em canonical divisor} of $X$.
The {\em Riemann-Roch theorem} gives a relation between
divisors of a Riemann surface $X$ and its topology, characterized by the genus $g=g(X)$: $l(D)-i(D)={\rm deg}(D)+1-g$, where
$l(D)$ is the dimension of the $\mathbb{C}$-vector space $L(D)=\{f\in\mathcal{M}(X)\,|\, (f)+D\ge 0\}$.
$i(D)$ is the dimension of the $\mathbb{C}$-vector space $I(D)=\{\omega\in\mathcal{M}^1(X)\, |\, (\omega)\ge D\}$.
In Tab. \ref{examples-applications-riemann-roch-theorem} are resumed some examples of applications of Riemann-Roch theorem.}
\end{lemma}
\begin{proof}
This result is standard. (See, e.g. \cite{GRIFFITHS-HARRIS, LIU, JOST}.)
\end{proof}

\begin{table}[t]
\caption{Some relations between divisors, spaces of meromorphic functions and topology.}
\label{examples-applications-riemann-roch-theorem}
\scalebox{0.8}{$\begin{tabular}{|c|c|c|c|c|c|}
\hline
\hfil{\rm{\footnotesize $D$}}\hfil&\hfil{\rm{\footnotesize ${\rm deg}(D)$}}\hfil&\hfil{\rm{\footnotesize $l(D)$}}\hfil&
\hfil{\rm{\footnotesize $L(D)$}}\hfil&\hfil{\rm{\footnotesize $i(D)$}}\hfil&\hfil{\rm{\footnotesize $I(D)$}}\hfil\\
\hline
\hfil{\rm{\footnotesize }}\hfil&\hfil{\rm{\footnotesize ${\rm deg}(D)<0$}}\hfil&\hfil{\rm{\footnotesize $l(D)=0$}}\hfil&
\hfil{\rm{\footnotesize $L(D)=\varnothing$}}\hfil&\hfil{\rm{\footnotesize }}\hfil&\hfil{\rm{\footnotesize }}\hfil\\
\hline
\hfil{\rm{\footnotesize $D=0$}}\hfil&\hfil{\rm{\footnotesize ${\rm deg}(D)=0$}}\hfil&\hfil{\rm{\footnotesize $l(D)=1$}}\hfil&
\hfil{\rm{\footnotesize $L(D)=C$}}\hfil&\hfil{\rm{\footnotesize $i(D)=g$}}\hfil&\hfil{\rm{\footnotesize $I(D)=L(K)$}}\hfil\\
\hline
\hfil{\rm{\footnotesize $D=K$}}\hfil&\hfil{\rm{\footnotesize ${\rm deg}(D)=-\chi(X)=2g-2$}}\hfil&\hfil{\rm{\footnotesize $l(D)=g=i(0)$}}\hfil&
\hfil{\rm{\footnotesize $L(D)=I(0)$}}\hfil&\hfil{\rm{\footnotesize $i(D)=1$}}\hfil&\hfil{\rm{\footnotesize $I(D)=L(0)$}}\hfil\\
\hline
\multicolumn{6}{l}{\rm{\footnotesize One has the isomorphisms: $I(D)\cong L(K-D)$;  $I(K-D)\cong L(D)$.}}\\
\multicolumn{6}{l}{\rm{\footnotesize If $D_1 \backsim D_2$ one has $L(D_1)\cong L(D_2)$; $I(D_1)\cong I(D_2)$.}}\\
\multicolumn{6}{l}{\rm{\footnotesize Riemann-Roch theorem: $l(D)-i(D)={\rm deg}(D)-1+g$.}}\\
\multicolumn{6}{l}{\rm{\footnotesize For $X=\hat S^1$ one has $K=-2\, q_\infty$.
${\rm deg}(K)=-2$, $l(K)=0$, $L(K)=\varnothing$, $I(K)=\mathbb{C}$, $i(K)=1$.}}\\
\multicolumn{6}{l}{\rm{\footnotesize [$\omega=dz$ in the south pole open disk $\triangle_-$ to which
correspondes]}}\\
\multicolumn{6}{l}{\rm{\footnotesize the form $\omega=d(\frac{1}{z})=-\frac{1}{z^2}dz$ in the north pole open disk $\triangle_+$,
since the transition map is $z\to \frac{1}{z}$]}}\\
\end{tabular}$}
\end{table}

\begin{lemma}\label{quantum-complex-zeta-riemann-function}
The completed zeta Riemann function $\tilde\zeta:\mathbb{C}\to\mathbb{C}$, identifies a surjective quantum mapping $\hat\zeta:\hat S^1\to\hat S^1$, with total branching index $b=2$, that we call {\em quantum-complex zeta Riemann function}. This is a meromorphic function between two Riemann spheres, with two simple poles and two simple zeros, and ${\rm deg}(\hat\zeta)=2$. $\hat\zeta$ identifies a ramified covering over $\hat S^1$, with two ramification points.
For all but the two branching points $w\in\hat S^1$, the equation $\hat\zeta(s)=w$ has exactly two solutions.
\end{lemma}
\begin{proof}
$\hat\zeta(s)$ is the unique meromorphic function $\hat\zeta:\hat S^1\to\hat S^1$ having the same poles at finite of $\tilde\zeta$, the two symmetric zeros of $\tilde\zeta$ lying on the critical line and nearest to the $x$-axis, and such that for a fixed $s_0\in\mathbb{C}$ one has satisfied condition  (\ref{definition-commutative-diagram-quantum-zeta-riemann-function}).
\begin{equation}\label{definition-commutative-diagram-quantum-zeta-riemann-function}
  \hat\zeta(s_0)=\tilde\zeta(s_0),\, s_0\in\mathbb{C}.
\end{equation}
More precisely the function $\hat\zeta(s)$ is a meromorphic function characterized by the divisor
\begin{equation}\label{divisor-hat-zeta}
    (\hat\zeta)=1\cdot q_{z_0}+1\cdot q_{-z_0}-1\cdot q_0-1\cdot q_1.
\end{equation}

 (See in Appendix A for an explicit proof.) Therefore, the process of Alexandrov compactification produces the reduction to only two zeros on the {\em critical circle}, $S^1\subset\hat S^1$, i.e., the compactified critical line, by an universal covering induced phenomena.\footnote{Let us note that $\tilde\zeta(s=\frac{1}{2}+it)\equiv \tilde\zeta(t)\in\mathbb{R}$, namely $\tilde\zeta(s)$ on the critical line is a real valued function. This follows directly from the functional equation (\ref{complted-functional-equation}). In fact, $\tilde\zeta(\frac{1}{2}+it)=\tilde\zeta(1-\frac{1}{2}-it)=\tilde\zeta(\frac{1}{2}-it)=\tilde\zeta(\overline{\frac{1}{2}+it})=\overline{\tilde\zeta(\frac{1}{2}+it)}$, hence $\Im(\tilde\zeta(\frac{1}{2}+it))=0$.}  (The situation is pictured in Fig. \ref{graph-completed-function-on-critical-line-and-corresponding-covering}.) The surjectivity of $\hat\zeta:\hat S^1\to\hat S^1$ follows from the following lemma.

\begin{lemma}\label{surjectivity-holomorphic-maps-between-compact-riemann-surfaces}
A non-constant holomorphic map between compact connected Riemann surfaces is surjective.
 \end{lemma}
 \begin{proof}
 The Proof is standard. (See, e.g., \cite{DONALDSON,JOST}.)
 \end{proof}
 Furthermore, since $\hat\zeta(s)$ is a rational function, $\hat\zeta(s)=\frac{p(s)}{q(s)}$, with ${\rm deg}(p)=2={\rm deg}(q)$, it follows that ${\rm deg}(\hat\zeta)=2$ and for all but finitely many $w$, the equation $\hat\zeta(s)=w$ has exactly $\max({\rm deg}(p),{\rm deg}(q))$ solutions.\footnote{For example, $\hat\zeta(s)=\infty$ has solutions $s\in\{0,1\}$, and $\hat\zeta(s)=0$ has solutions $s\in\{-z_0,z_0\}$.} The total branching index $b$ of $\hat\zeta$ can be calculated by using the following lemma.

 \begin{lemma}[Riemann-Hurwitz formula]\label{riemann-hurwitz-formula}
Let $f:R\to S$ be a non-constant holomorphic map between compact connected Riemann surfaces. Then the following algebraic topologic formula holds.
\begin{equation}\label{riemann-hurwitz-formula-equation}
\chi(R)={\rm deg}(f)\, \chi(S)-b
\end{equation}
where $\chi(X)=2-2\, g$ is the Euler characteristic of any compact, connected, orientable surface $X$, without boundary, and $g$ denotes its genus: $g=\frac{1}{2}{\rm dim}_{\mathbb{C}}H_1(X;\mathbb{C})$.\footnote{From this formula it follows that $b$ must be an even integer too. Furthermore, this lemma can be generalized to non-compact Riemann surfaces, such that $f$ is a proper map and $\chi(S)$ is finite. (As a by product it follows that also $\chi(R)$ is finite.) (Let us recall that $f$ is {\em proper} iff $f^{-1}$ of a compact set is compact.)}
\end{lemma}
 \begin{proof}
 The Proof is standard. (See, e.g., \cite{DONALDSON,HIRZEBRUCH,JOST})
 \end{proof}
 Since $\chi(\hat S^1)=2$ it follows that $b=2\cdot 2-2=2$. In order to identify where the ramifications points are located, let us find the roots of the equation $\hat\zeta'(s)=0$. We get, $\hat\zeta'(s)=c_0\frac{-2(\frac{1}{4}+\alpha^2)s+(\frac{1}{4}+\alpha^2)}{s^2(s-1)^2}$, with $\alpha=\Im(z_1)=-\Im(z_2)$. Therefore the unique critical point at finite is $s=\frac{1}{2}$, with multiplicity $e_{\frac{1}{2}}=2$, since $\hat\zeta^{(2)}(\frac{1}{2})\not=0$. From the Riemann-Hurwitz theorem we get that also the point $\infty$ is a ramification point with ramification index $e_{\infty}=2$. See the commutative diagram (\ref{riemann-hurwitz-diagram-quantum-complex-zeta-riemann}).
 \begin{equation}\label{riemann-hurwitz-diagram-quantum-complex-zeta-riemann}
    \xymatrix@C=0.5cm{\chi(\hat S^1)\ar@{=}[d]\ar@{=}[r]&2\cdot {\rm deg}(\hat\zeta)-[(e_\frac{1}{2}-1)+(e_\infty-1)]\ar@{=}[d]\\
    2\ar@{=}[r]&4-[1+1]\\}
 \end{equation}

 Removing from $\hat S^1$ the branch points $\hat\zeta(\infty),\, \hat\zeta(\frac{1}{2})\in\hat S^1$, we get a punctured sphere. The number of the points in the inverse image $\hat\zeta^{-1}(q)$, for $q$ belonging to such a punctured sphere,\footnote{This punctured sphere is the parabolic Riemann surface called {\em annulus} or {\em cylinder}.} is integer-valued and continuous, hence constant. It coincides with the degree of $\hat\zeta$, namely ${\rm deg}(\hat\zeta)=2$.
Therefore $\hat\zeta$ identifies a ramified cover of $\hat S^1$, with two ramification points $s=\frac{1}{2}$ and $s=\infty$, and two branching points $\hat\zeta(\frac{1}{2})$ and $\hat\zeta(\infty)$ respectively. After removing the two branch points and their two preimages, $\hat\zeta$ induces a double topological covering map.\footnote{Recall that the {\em degree of a branched cover} coincides with the degree of the induced covering map, after removing the branch points and the corresponding ramifications points.}
\end{proof}

 \begin{figure}[h]
 \includegraphics[height=3cm]{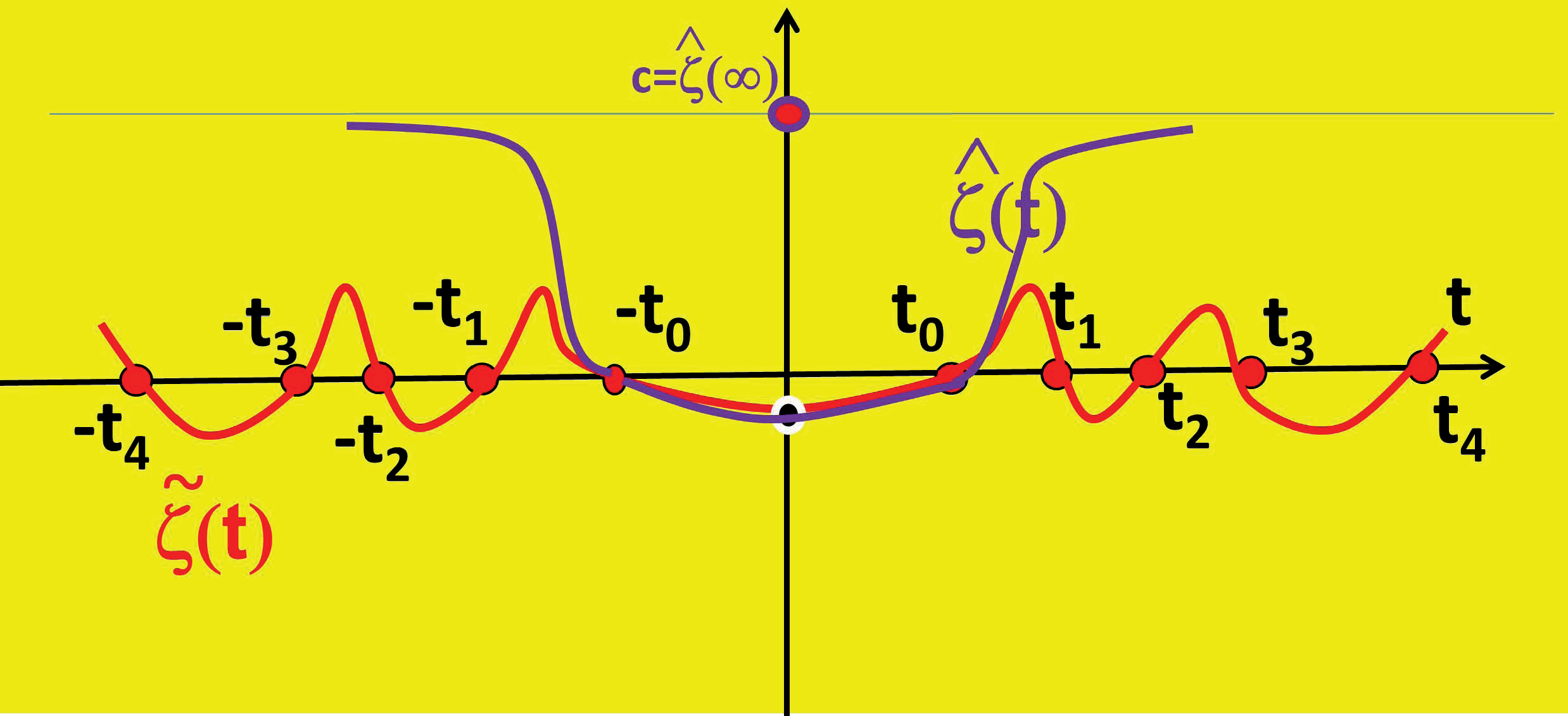}
 \includegraphics[height=3cm]{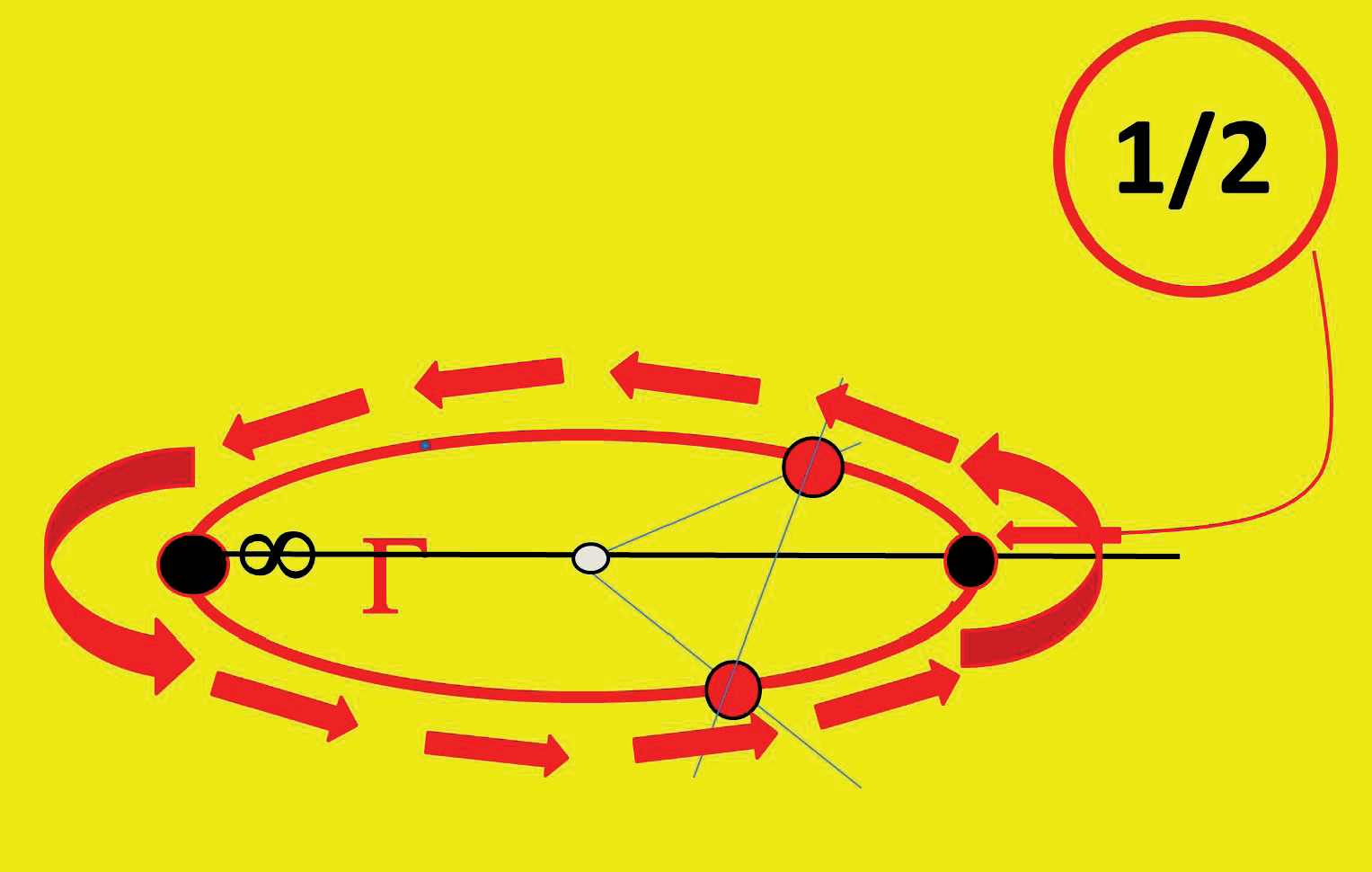}
\renewcommand{\figurename}{Fig.}
\caption{Representation of the numerical functions $\tilde\zeta:\mathbb{R}\to\mathbb{R}$ and $\hat\zeta:\mathbb{R}\to\mathbb{R}$, restrictions to the critical line of $\tilde\zeta$ and $\hat\zeta$ respectively, (figure on the left), There is reported also the point base $\tilde\zeta(t=0)=\tilde\zeta(\frac{1}{2})=-0.05438=\hat\zeta(\frac{1}{2})$, (withe-circled-black point). The red circle $\Gamma$, in the figure on the right, represents the compactification to $\infty$, of the critical line. There the other black point, emphasized with the circled $\frac{1}{2}$, is the point corresponding to $s=\frac{1}{2}$ on the $x$-axis in $\mathbb{C}$, ($\Im(s)=0$). In this picture the red arrows circulation mean the covering of the critical line in such a way to project all zeros of $\tilde\zeta$ on the two zeros of $\hat\zeta$. (See in the proof of Lemma \ref{quantum-complex-riemann-zeta-covering}.)}
\label{graph-completed-function-on-critical-line-and-corresponding-covering}
\end{figure}

\begin{definition}\label{quantum-extension}
Let $\mathfrak{Q}_{\mathbb{C}}$ be the category of quantum-complex manifolds. Let $f:\mathbb{C}\to\hat S^1$ be a morphism in this category, such that there exists a finite number of points $p_k\in\mathbb{C}$, $k=1,\cdots, n$, with $f(p_k)=\infty$.\footnote{In other words $f$ is a meromorphic function with poles $p_k$.} We call {\em quantum-extension} of $f$, a morphism $\hat f:\hat S^1\to\hat S^1$, in the category $\mathfrak{Q}_{\mathbb{C}}$, such that $\hat f(\hat p_k)=\infty$, where $\hat p_k\in \hat S^1$ are points corresponding to $p_k$, via the stereographic projection $\pi_\star:\hat S^1\to\mathbb{C}$. There are not other points $q\in\hat S^1$ such that $\hat f(q)=\infty$. Furthermore, if $f(z_k)=0$, then $\hat f(\hat z_k)=0$ for a finite number of zeros $z_k$ of $f$, that are nearest to the $x$-axis. Here $\hat z_k$ is the point on $\hat S^1$, corresponding to $z_k$, via the stereographic projection $\pi_\star$.
\end{definition}
\begin{definition}\label{pointed-quantum-extension}
We call {\em pointed quantum-extension} of $f\in Hom_{\mathfrak{Q}_{\mathbb{C}}}(\mathbb{C},\hat S^1)$ a quantum extension $\hat f$ such that $\hat f(\hat q_0)=f(q_0)$, where $q_0\in\mathbb{C}$ is a fixed point and $\hat q_0\in\hat S^1$ is the point on $\hat S^1$, corresponding to $q_0$ via the stereographic projection. Then we say also that the quantum-extension $\hat f$ is {\em pointed} at $q_0$.
\end{definition}

\begin{lemma}[Uniqueness of quantum-extension of completed Riemann zeta function]\label{uniqueness-quantum-extension-completed-riemann-zeta-function}
The quantum-complex zeta Riemann function $\hat\zeta$ is the unique quantum-extension of $\tilde\zeta$, pointed at $q_0=\frac{1}{2}$.
\end{lemma}
\begin{proof}
From above results it follows that $\hat\zeta$ is surely a pointed quantum-extension of $\tilde\zeta$ (pointed at $q_0=\frac{1}{2}$.) The uniqueness follows from Lemma A\hskip-0.5pt\ref{appendix-a-lemma-c} in Appendix A. In fact, the pointing fixes the arbitrary constant $c\in\mathbb{C}$, present there in equation (\ref{appendix-a-equation-appendix-lemma-c}). Let us also emphasize that we could consider other zeros, outside the critical line, but yet in the critical strip, if they should exist. These should be at sets of four, symmetrically placed with respect to the $x$-axis and with respect to the critical line. In such a case $\hat\zeta$ could not more have only two simple poles at $0$ and $1$, but it should have also a double point at $\infty$. In fact all such points should have the same distance with respect to the $x$-axis. But in this case $\hat\zeta$ could not be a quantum-extension of $\tilde\zeta$, in the sense of Definition \ref{quantum-extension}. Below and in appendices it is also carefully proved why such eventual zeros $\tilde\zeta$, outside the critical line, cannot exist. Therefore the quantum-complex zeta Riemann function $\hat\zeta$, considered in Lemma \ref{quantum-complex-zeta-riemann-function}, is the unique pointed quantum-extension of $\tilde\zeta$.
\end{proof}

\begin{lemma}\label{quantum-complex-riemann-zeta-covering}
The completed Riemann zeta function $\tilde\zeta:\mathbb{C}\to\mathbb{C}$ identifies a surjective meromorphic mapping $\tilde\zeta:\mathbb{C}\to\hat S^1$, (hence yet denoted $\tilde\zeta$). Furthermore, we can identify a continuous surjective mapping $\hat b:\mathbb{C}\to\hat S^1$, such that all zeros $\pm z_k$ on the critical line of $\tilde\zeta$, are projected to $\pm z_0$, i.e., the zeros of $\hat\zeta$, belonging to the critical circle $\Gamma\subset\hat S^1$ and nearest to the x-axis or equator.
\end{lemma}

\begin{proof}
Since $\tilde\zeta$ is a meromorphic function on $\mathbb{C}$, with simple poles in $0,\, 1\in \mathbb{C}$, it identifies an holomorphic function (yet denoted $\tilde\zeta$):
 \begin{equation}\label{riemann-sphere-valued-composition}
    \xymatrix{\mathbb{C}\ar@/_1pc/[rr]_{\tilde\zeta}\ar[r]^{\tilde\zeta}&\mathbb{C}\ar@{^{(}->}[r]^(0.3){\hat i}&\mathbb{C}\bigcup\{\infty\}\\}
 \end{equation}
 The surjectivity of the map $\tilde\zeta:\mathbb{C}\to\hat S^1$ follows directly from its holomorphy. In fact this map is open
 and continuous.  Then let $D_r$ denote a closed disk in $\mathbb{C}\cong \mathbb{R}^2$, centered at $0$ and of radius $r$, such that
 it contains other the point $0$ also the points $1$ and the zeros $\pm z_0$ on the critical line.
 Then $\tilde\zeta(D_r)$ is open in $\hat S^1\cong S^2$ and compact, hence closed. Furthermore, $\tilde\zeta(D_r)$ contains the
 points $0$, $\infty$ and $\tilde\zeta(\frac{1}{2})$ that is placed on the meridian on $S^2$ corresponding to the $x$-axis in
 $\mathbb{R}^2$. Therefore, $\tilde\zeta(D_r)$ is necessarily a compact surface in $S^2$ with boundary a circle $\Gamma_r$.
 When $r\to\infty$, $\Gamma_r$ converges for continuity to a point $\maltese\in S^2$. In fact, let us assume that the sequence
 \begin{equation}\label{sequences-filtering-image-completed-riemann-zeta-function}
    \tilde\zeta(D_r)\subseteq \tilde\zeta(D_{r'})\subseteq\tilde\zeta(D_{r''})\subseteq \tilde\zeta(D_{r'''})\subseteq\cdots,\,
    r<r'<r''<r'''<\cdots
 \end{equation}
 stabilizes, i.e., there exists a maximal disk $D_{r_{max}}\subset\mathbb{R}^2$, such that
 $\tilde\zeta(D_{s})=\tilde\zeta(D_{r_{max}})$, for any $D_s\supset D_{r_{max}}$. Then there exists also an open disk $\triangle\subset\hat S^1$
 such that $ \tilde\zeta(D_{r_{max}})\subset \triangle$. Thus we can state that $\tilde\zeta$ is a global holomorphic bounded mapping
 $\tilde\zeta:\mathbb{C}\to\triangle$. But such a mapping must be constant for the {\em Liouville's theorem} in complex analysis. Therefore,
 the hypothesis that the sequence (\ref{sequences-filtering-image-completed-riemann-zeta-function}) stabilizes is in contradiction with
 the fact that $\tilde\zeta$ is not a constant map. By conclusion $\tilde\zeta:\mathbb{C}\to\hat S^1$ must be a surjective map.
 Furthermore, one has a continuous mapping $\hat b:\mathbb{C}\to\hat S^1$, sending the zeros of $\tilde\zeta$, on the critical line, to cover the two zeros of $\hat\zeta$. In fact, let us first consider the mapping $\hat a:\mathbb{C}\to\hat S^1$, defined by $\hat a:(z=x+i\, y)\mapsto(x_0,e^{i2\pi y})$, where $x_0$ is identified by the stereographic projection as a point of the circle $\hat S^1_{x}\subset\hat S^1$, corresponding in the stereographic projection to the $x$-axis of $\mathbb{C}$. Let $\Gamma_{x_0}$ be the circle in $\hat S^1$, passing for $\infty$ and $x_0$ and having its plane parallel to the $y$-axis. Then starting from the vector $\overrightarrow{C_{x_0}x_0}$, in the plane of $\Gamma_{x_0}$, with center $C_{x_0}$, we calculate the angle $2\pi\, y$ that identifies a point on $\Gamma_{x_0}$. In this way we realize a covering $\hat a:\mathbb{C}\to\hat S^1$ with fiber $\mathbb{Z}$. In fact, for any point $p\in\hat S^1$ we can identify the corresponding circle $\Gamma_{x_0}$ passing for $p$, hence the point $x$ on the $x$-axis of $\mathbb{C}$. Furthermore, $p$ identifies an angle $\alpha=2\pi\, y+2\pi n$, hence $y=\frac{\alpha}{2\pi}-n$, $n\in\mathbb{Z}$. For continuity we put $|\hat a^{-1}(p)|=\infty$. This allows us to identify the covering $\mathbb{Z}\to\mathbb{C}\to\hat S^1$.  Let us now consider the real function $\tilde\zeta(t)$ obtained by restriction of $\tilde\zeta$ on the critical line. With this respect, we define the following continuous mapping $\phi:\mathbb{R}^2\to\mathbb{R}^2$, $(x,y)\mapsto(\bar x, \bar y)=(x,f(y))$, such that $f(y)$ is obtained by sectorial retraction (dilation) in order to deform all intervals $z_1-z_0$, $z_2-z_1,\, \cdots, z_k-z_{k-1},\,\cdots$ in ones of the same length $2\pi$. In this way the composition mapping $\hat b=\hat a\circ \phi:\mathbb{C}\to\hat S^1$ identifies a continuous mapping that projects all the zero $\pm z_k$ of $\tilde\zeta$ on the same couple $\pm z_0$ respectively, of $\Gamma=\Gamma_{(\frac{1}{2})}$. The mapping $\hat b$ is continuous but not holomorphic. (See in Appendix B for an explicit proof.)\footnote{For example if $(x,y)\in \mathbb{R}^2$, with $z_0\le y\le z_1$, then the transformation of $\mathbb{R}^2$, in this sector (or strip), is $(x,y)\mapsto (x,\bar y)$, with $\bar y=z_0+\frac{y-z_0}{z_1-z_0}$. Therefore, the points of this sector are projected on $(x_0,e^{i2\pi\bar
y})=(x_0,e^{i2\pi(z_0+\frac{y-z_0}{z_1-z_0}})$.} We call $\hat b$ the {\em zeros-Riemann-zeta-normalized universal covering} of $\hat S^1$.\footnote{Let us emphasize that this mapping $\hat b$ can be related in some sense to the Weierstrass-$\wp$ function. In fact, a surjection $\mathbb{C}\to\hat S^1$ can be identified by utilizing the projection of $\mathbb{C}$ on the torus $\mathbb{C}/L$, where $L=\mathbb{Z}\, \omega_1+\mathbb{Z}\, \omega_2\subset\mathbb{C}$ is a lattice with $\omega_1,\, \omega_2\not=0$, $\omega_1/\omega_2\not\in\mathbb{R}$. In fact the quotient $(\mathbb{C}/L)/\{\pm 1\}$ can be identified with $\hat S^1$. In other words, one has the exact commutative diagram (\ref{commutative-diagram-quantum-zeta-riemann-function}). There it is also emphasized the relation between the torus $T=S^1\times S^1$ and the smash product $S^1\wedge S^1=(S^1\times S^1)/S^1\vee S^1\backsimeq S^2$. We can consider the wedge sum $S^1\vee S^1=(S^1\times\{a\}\bigcup S^1\times\{b\})/\{a,b\}$ reduced to $\infty$ in $S^2$. It is important to emphasize also that any continuous mapping $f:\mathbb{C}/L\to\hat S^1$, can be factorized in the form $f=h\circ \wp$, where $\wp:\mathbb{C}/L\to\hat S^1$ is the Weierstrass $\wp$-function and $h:\hat S^1\to\hat S^1$ is some continuous mapping. Recall that $\wp$ is defined by the following series: $\wp(u)=\frac{1}{u^2}+\sum_{\omega\in L^\bullet}[\frac{1}{(u-\omega)^2}-\frac{1}{\omega^2}]$, converging to an {\em elliptic function}, i.e., a doubly periodic meromorphic function on $\mathbb{C}$, hence a meromorphic function on a torus $\mathbb{C}/L$. (A doubly periodic holomorphic function on $\mathbb{C}$ is constant.) $\wp$ is a degree $2$ holomorphic map with branch points over $\infty,\, e_k$, $k=1,2,3$, where $e_k=\wp(u_k)$, $u_k\in\{\frac{\omega_1}{2},\frac{\omega_2}{2},\frac{\omega_1+\omega_2}{2}\}$. In other words, $\wp$ has the following four ramification points: $\{0,u_k\}_{1\le k\le 3}\subset \mathbb{C}/L$. For any $w\in\mathbb{C}$. the equation $\wp(u)=w$ has two simple roots in the period parallelogram; instead for $w=e_k$, the equation $\wp(u)=w$ has a single double root. Every elliptic function $f(u)$ can be expressed uniquely by means of $\wp$ and its first derivative $\wp'$, as follows: $f(u)=R_0(\wp(u))+\wp'(u)\, R_1(\wp(u))$, where $R_i$, $i=0,1$, denote rational functions. If $f(u)$ is even (resp. odd) then in the above expression compares only the term with the function $R_0$ (resp. $R_1$).}
\end{proof}

\begin{equation}\label{commutative-diagram-quantum-zeta-riemann-function}
\scalebox{0.7}{$
\xymatrix{
  0&\ar[l]\mathbb{R}^2/\mathbb{Z}^2\ar@{=}^{\wr}[dddd]&\ar[l]\mathbb{R}^2\ar@{=}[d]\ar@{=}^{\sim}[dr]&&&\\
  &&\mathbb{R}\times\mathbb{R}\ar[d]&\mathbb{C}\ar[d]\ar[r]^{\hat b}&\hat S^1\ar[r]\ar@{-}^{\wr}[d]&0\\
  &&(\mathbb{R}/\mathbb{Z})\times(\mathbb{R}/\mathbb{Z})\ar@{=}[dd]&\mathbb{C}/L\ar[r]\ar[d]&(\mathbb{C}/L)/\{\pm 1\}\ar@{-}^{\wr}[d]\ar[r]&0\\
  &&&0&S^2\ar@{=}^{\wr}[d]&\\
  &T\ar@{=}[r]&S^1\times S^1\ar[d]\ar@{=}^(0.4){\sim}[uur]\ar[d]\ar[rr]&&S^1\wedge S^1\ar[r]&0\\
  &&0&&&\\}$}
\end{equation}

From Lemma~\ref{quantum-complex-zeta-riemann-function} and Lemma~\ref{quantum-complex-riemann-zeta-covering} we can state that passing from the function $\tilde\zeta$ to $\hat\zeta$, all zeros of $\tilde\zeta$ on the critical line converge to two zeros only. Furthermore, no further zeros can have $\tilde\zeta$ outside the critical line, otherwise they should converge to some other zeros of $\hat\zeta$, outside the critical line. But such zeros of $\hat\zeta$ cannot exist since all possible zeros of $\hat\zeta$ are reduced to only two, and these are just the ones considered on the critical circle $\Gamma$. We can resume the relation between $\tilde\zeta$ and $\hat\zeta$ also by means of the exact commutative diagram (\ref{commutative-diagram-quantum-zeta-riemann-function-a}) showing that $\hat\zeta$ is the (unique) meromorphic function on $\hat S^1$ on the which one can project $\tilde\zeta$ by means of the compactification process above considered.
 \begin{equation}\label{commutative-diagram-quantum-zeta-riemann-function-a}
\scalebox{0.8}{$\xymatrix@C=1cm{&\mathbb{C}\ar[dr]^{\hat a}&&\\
  \mathbb{C}\ar[ur]^{\phi}\ar[d]_{\tilde \zeta}\ar[rr]^{\hat b}&&
  \hat S^1\ar[d]_{\hat\zeta}\ar[r]&0\\
  \hat S^1\ar[d]\ar[rr]_{\hat b'}&&\hat S^1\ar[d]\ar[r]&0\\
  0&&0&\\}$}
\end{equation}
By resuming, we can say that the mapping $\hat b':\hat S^1\to\hat S^1$, in (\ref{commutative-diagram-quantum-zeta-riemann-function-a}), is uniquely identified with a continuous mapping such that diagram (\ref{commutative-diagram-quantum-zeta-riemann-function-a}) is commutative. Since $\hat b$ is continuous but not holomorphic, it follows that also $\hat b'$ is continuous but not holomorphic Thus the couple $(\hat b,\hat b')$ is a continuous but not holomorphic, it is not a morphism of the category $\mathfrak{Q}_{\mathbb{C}}$, but it is a morphism in the category $\mathfrak{T}$ of topological manifolds  that contains $\mathfrak{Q}_{\mathbb{C}}$. Therefore $(\hat b,\hat b')$ deforms the morphisms $\tilde\zeta$ and $\hat\zeta$ of $\mathfrak{Q}_{\mathbb{C}}$, each into the other.
\begin{remark}\label{remark-zeros-outside-critical-line}
Let us also emphasize that whether some zeros of $\tilde\zeta$ should exist outside the critical line, it should be impossible continuously project $\tilde\zeta$ onto $\hat\zeta$ by means of the morphism $\hat\beta=(\hat b,\hat b')$. In fact, in such a case the mapping $\hat b'$ could not be surjective, since, whether $\tilde\zeta(s_k)=0$, for some zero $s_k$ outside the critical line, then $\hat\zeta(\hat b(s_k))\not=0$. Therefore, $\hat b'$ should map $0$ to $\{\hat\zeta(s_k)\not=0,0\}$, hence $\hat b'$ should be multivalued. Therefore, in order that $\tilde\zeta$ projects on $\hat\zeta$,  by means of the morphism $\hat\beta$, it is necessary that $\tilde\zeta$ has not zeros outside the critical line.
\end{remark}

Taking into account above results it is useful to introduce now some definitions and results showing a relation betweeen the proof of the Riemann hypothesis and an important algebraic topologic property of the special functions involved. In fact, one has the following lemma.

\begin{lemma}[Bordism of completed Riemann zeta function]\label{bordism-morphism-and-quantum-riemann-zeta-function}
The quantum-complex zeta Riemann function $\hat\zeta$ canonically identifies a quantum-complex bordism element, $[\hat\zeta]\in\hat\Omega_1(\tilde\zeta)$, even if $\tilde\zeta$ has some zeros outside the critical line. Here $\hat\Omega_1(\tilde\zeta)$ denotes the $1$-quantum-complex bordism group of the completed Riemann zeta function.
\end{lemma}
\begin{proof}
Let us first consider some definitions and results related to bordism of morphisms.\footnote{The subject considered in this lemma is a generalization to the category $\mathfrak{Q}_{\mathbb{C}}$ of some particular bordism groups introduced by R. E. Stong \cite{STONG}. (See, also \cite{ATIYAH} for related subjects.)} Let $\mathfrak{T}$ be the category of topological manifolds and $f\in Hom_{\mathfrak{T}}(X,Y)$, $f'\in Hom_{\mathfrak{T}}(X',Y')$, two morphisms in the category $\mathfrak{T}$, in the following called {\em maps}. Let us also define {\em morphism} $\alpha:f\to f'$ a pair $\alpha=(a,a')$, with $a\in Hom_{\mathfrak{T}}(X;X')$, $a'\in Hom_{\mathfrak{T}}(Y;Y')$, such that the diagram (\ref{commutative-diagram-morphism}) is commutative.
\begin{equation}\label{commutative-diagram-morphism}
\xymatrix@C=2cm{X\ar[d]_{a}\ar[r]^{f}&Y\ar[d]_{a'}\\
X'\ar[r]_{f'}&Y'\\}
\end{equation}
In particular we say that $\alpha$ {\em projects} the morphism $f$ on $f'$, if the commutative diagram
(\ref{commutative-diagram-morphism-surjective}) is exact.
\begin{equation}\label{commutative-diagram-morphism-surjective}
\xymatrix@C=2cm{X\ar[d]_{a}\ar[r]^{f}&Y\ar[d]_{a'}\\
X'\ar[d]\ar[r]_{f'}&Y'\ar[d]\\
0&0\\}
\end{equation}
In other words, $\alpha$ is surjective if $a$ and $a'$ are so. In such a case we write $\alpha:f\to f'\to 0$.
A {\em quantum-complex map} is a morphism $\hat f$ in the subcategory $\mathfrak{Q}_{\mathbb{C}}\subset\mathfrak{T}$,
$\hat f\in Hom_{\mathfrak{Q}_{\mathbb{C}}}(\hat X;\hat Y)$ such that $\hat X$ and $\hat Y$ are compact quantum-complex manifolds
and $\hat f(\partial\hat X)\subset\partial\hat Y$. A {\em closed quantum-complex map} is a  quantum-complex map, with
 $\partial \hat X=\partial\hat Y=\varnothing$. A {\em quantum-complex bordism element} of a map $f\in Hom_{\mathfrak{T}}(X,Y)$
 is an equivalence class of morphisms $\alpha:\hat f\to f$, where $\hat f$ is a closed quantum-complex map.
 Two morphisms $\alpha_i=(a_i,a'_i):(\hat\phi_i,\hat M_i,\hat N_i)\to f$, $i=1,2$, are equivalent if there is a morphism
 $\alpha=(a,a'):\hat\Phi=(\hat\phi,\hat V,\hat W)\to f$, with $\hat\Phi$ a quantum-complex map such that
 $\partial \hat V=\hat M_1\sqcup\hat M_2$, $\partial\hat W=\hat N_1\sqcup\hat N_2$,
 $\hat\phi|_{\hat M_i}=\hat\phi_i$, $a|_{\hat M_i}= a_i$, $a'|_{\hat N_i}=a'_i$. The set of quantum-complex bordism elements
 of $f\in Hom_{\mathfrak{T}}(X,Y)$ forms an abelian group $\hat\Omega_\bullet(f)$, ({\em quantum-complex bordism group of $f$}),
 with the operation induced by the disjoint union; i.e., $[\alpha_1]+[\alpha_2]=[\alpha_1\bigcup\alpha_2]$, where
 $\alpha_1\bigcup\alpha_2=(a_1\bigcup a_2,a'_1\bigcup a'_2)$, with
 $a_1\bigcup a_2:\hat M_1\sqcup \hat M_2\to X$, $a'_1\bigcup a'_2:\hat N_1\sqcup \hat N_2\to Y$ are the obvious maps on the
 disjoint unions. We call {\em $n$-dimensional quantum-complex bordism group of $f$}, when
 $\dim_{\mathbb{C}}\hat M=\dim_{\mathbb{C}}\hat N=n$. Let us, now specialize to the completed Riemann zeta function
 $\tilde\zeta$. One has that the diagram (\ref{commutative-diagram-morphism-completed-riemann-zeta-function}) is
 commutative.
\begin{equation}\label{commutative-diagram-morphism-completed-riemann-zeta-function}
\xymatrix@C=2cm{0&0&\\
\mathbb{C}\ar[u]\ar[r]^{\tilde\zeta}&\hat S^1\ar[u] \ar[r]&0\\
\hat S^1\ar[u]_{\pi_\star}\ar[r]_{\hat\zeta}&\hat S^1\ar[u]^{\pi'_\star}\ar[r]&0\\}
\end{equation}
where $\pi_\star$ is the stereographic projection and $\pi'_\star$ is uniquely defined. In particular $\pi'_\star$ is surjective.\footnote{Let us emphasize that the composed mapping $\hat f=\tilde\zeta\circ\pi_\star:\hat S^1\to\hat S^1$ is not holomorphic or meromorphic. Therefore it cannot be a quantum-complex zeta Riemann function. See Lemma \ref{trivial-extension} and Appendix C for a detailed proof.} Therefore diagram (\ref{commutative-diagram-morphism-completed-riemann-zeta-function}) is also exact. As a by-product we get that $\hat\zeta$ identifies a quantum-complex bordism element of $\tilde\zeta$, i.e., a morphism $\hat f\to\tilde\zeta$ is equivalent to $\hat\zeta\to\tilde\zeta$ if there exists a morphism $\alpha=(c,c'):\hat\Phi=(\hat\phi,\hat V,\hat W)\to\tilde\zeta$, such that $\partial\hat V=\hat S^1\sqcup \hat M$, $\partial\hat W=\hat S^1\sqcup \hat N$, $\hat\phi|_{\partial\hat V}=\hat\zeta\sqcup\hat f$,  $c|_{\partial\hat V}=\pi_\star\sqcup b$, $c'|_{\partial\hat W}=a'\sqcup \pi'_\star$. The situation is resumed in the commutative diagram (\ref{commutative-diagram-bordism-map-zeta-riemann-function}).
\begin{equation}\label{commutative-diagram-bordism-map-zeta-riemann-function}
\xymatrix@C=2cm{\mathbb{C}\ar[r]^{\tilde\zeta}&\hat S^1\\
\hat S^1\ar@{^{(}->}[d]\ar[u]_{\pi_\star}\ar[r]_{\hat\zeta}&\hat S^1\ar@{^{(}->}[d]\ar[u]^{\pi'_\star}\\
\hat V\ar@/^1pc/[uu]^{c}\ar[r]^{\hat\phi}&\hat W\ar@/_1pc/[uu]_{c'}\\
\partial \hat V=\hat S^1\sqcup\hat M\ar@/^3pc/[uuu]^{\pi_\star\cup b}\ar@{^{(}->}[u]\ar[r]_{\hat\phi|_{\partial\hat V}}&\partial\hat W=\hat S^1\sqcup\hat N\ar@{^{(}->}[u]\ar@/_3pc/[uuu]_{a\cup b'}\\
\hat M\ar@{^{(}->}[u]\ar@/^6pc/[uuuu]^{b}\ar[r]_{\hat f}&\hat N\ar@{^{(}->}[u]\ar@/_6pc/[uuuu]_{b'}\\}
\end{equation}

We denote by $[\hat\zeta]$ the equivalence class identified by $\hat\zeta$. Therefore one has $[\hat\zeta
]\in\hat\Omega_\bullet(\tilde\zeta)$. In other words, the exact and commutative diagram (\ref{commutative-diagram-morphism-completed-riemann-zeta-function}) says that the $\hat\zeta$ identifies a quantum-complex bordism element of $\tilde\zeta$, and that $\hat\zeta$ projects on $\tilde\zeta$.

$\bullet$\hskip 2pt $\hat\Omega(-)$ identifies a covariant functor on the category of maps and morphisms to the category of abelian groups and their morphisms. In other words, for any morphism $\alpha=(a,a'):f'\to f$ in $\mathfrak{T}$, with $f,\, f'\in Hom_{\mathfrak{Q}_{\mathbb{C}}}(X,Y)$, one has induced a homomorphism $\hat\Omega_\bullet(\alpha):\hat\Omega_\bullet(f')\to\hat\Omega_\bullet(f)$, by assigning to $\beta=(b,b'):f''\to f'$ the class $\alpha\circ\beta:f''\to f$, i.e., the diagram (\ref{commutative-diagram-functor-bordism}) is commutative.
\begin{equation}\label{commutative-diagram-functor-bordism}
\scalebox{0.6}{$\xymatrix@C=2cm{X\ar[rr]^{f}&&Y\\
&\circlearrowleft\alpha&\\
X'\ar[uu]^{a}\ar[rr]_{f'}&&Y'\ar[uu]^{a'}\\
&\circlearrowleft\beta&\\
X''\ar@/^3pc/[uuuu]\ar[uu]^{b}\ar[rr]_{f''}&&Y''\ar[uu]^{b'}\ar@/_3pc/[uuuu]\\}$}
\end{equation}

$\bullet$\hskip 2pt Furthermore, homotopic morphisms induce the same homomorphisms on the bordism groups. More precisely, if
 \begin{equation}\label{extension-homotopy-a}
 \left\{\begin{array}{l}
  \alpha=(a,a'):(\phi\times 1:X'\times[0,1]\to Y'\times[0,1])\to(f:X\to Y)\\
   \alpha_t=(a_t,a'_t):(\phi:X'\to Y')\to f, \hskip 2pt (a_t(x)=a(x,t),\, a'_t(y)=a'(t,y))\\
 \end{array}\right.
 \end{equation}
 namely the diagram (\ref{commutative-diagram-momotopic-morphism})
\begin{equation}\label{commutative-diagram-momotopic-morphism}
\xymatrix@C=2cm{X\ar[r]^{f}&Y\\
X'\times I\ar[u]^{a}\ar[r]_{\phi\times 1}&Y'\times I\ar[u]_{a'}\\}
\end{equation}
 is commutative, then $\hat\Omega_\bullet(\alpha_0)\cong\hat\Omega_\bullet(\alpha_1)$.

 It is important to underline that whether $\tilde\zeta$ should have some zero outside the critical line diagram (\ref{commutative-diagram-morphism-completed-riemann-zeta-function}) continues to be commutative, whether $\hat\zeta$ is the quantum-extension of the completed Riemann zeta function, i.e., its divisor is given by (\ref{divisor-hat-zeta}).
\end{proof}

\begin{lemma}\label{relations-bordism-groups}
Relations between the bordism groups $\hat\Omega_\bullet(\hat\zeta)$ and $\hat\Omega_\bullet(\tilde\zeta)$ are given by the diagram {\em(\ref{sequence-relation-bordism-groups-generalized})}.
\begin{equation}\label{sequence-relation-bordism-groups-generalized}
    \xymatrix{\hat\Omega_\bullet(\hat\zeta)\ar[r]^{\hat\alpha_\bullet}&\hat\Omega_\bullet(\tilde\zeta)\ar[d]_{\hat\beta_\bullet}\\
   & \hat\Omega_\bullet(\hat\zeta)\\}
\end{equation}
The homomorphism $\hat\beta_\bullet$ is subordinated to the condition that $\tilde\zeta$ has not zeros outside the critical line.
\end{lemma}
\begin{proof}
This follows directly from above definitions and results when one considers morphisms $\alpha=(\pi_*,\pi'_*):\hat\zeta\to\tilde\zeta$ and morphisms $\hat\beta=(\hat b,\hat b'):\tilde\zeta\to\hat\zeta$. This last is conditioned by the absence of zeros for $\tilde\zeta$, outside the critical line.
\end{proof}

With this respect let us introduce the following definitions.

\begin{definition}[Stability in $\hat\Omega_\bullet(f)$]\label{stability-b}
Let $\beta=(b,b')$, be a morphism with $b,\, b'\in Hom(\mathfrak{T})$, such that $\beta:f\to \hat f\to 0$, where $\hat f\in Hom_{\mathfrak{Q}_{\mathbb{C}}}(\hat X,\hat Y)$ is a closed map. We say that $\hat f$ is {\em $\beta$-stable} in $\hat\Omega_\bullet(f)$ if the following condition is satisfied.

{\em(i)} There exists a morphism $\alpha=(a,a')$, with $a,\, a'\in Hom(\mathfrak{T})$, such that $\alpha:\hat f\to f\to 0$.\footnote{Warning. Condition (i) is stronger than to state that $\hat f$ identifies a bordism class in $\hat\Omega_\bullet(f)$.}

Then we say also that $\hat f$ is {\em stable} in $\hat \Omega_\bullet(f)$, with respect to the {\em perturbation} $\beta$ of $f$.
\end{definition}
\begin{lemma}\label{stability-lemma-a}
If the map $\hat f\in Hom_{\mathfrak{Q}_{\mathbb{C}}}$ is $\beta$-stable in $\hat\Omega_\bullet(f)$ then there exists a nontrivial morphism $\gamma=(c,c')$, with $c,\, c'\in Hom(\mathfrak{T})$, such that diagram {\em(\ref{commutative-exact-diagram-stability-a})} is commutative and exact. (Here nontrivial means that $c\not=1_{\hat X}$ and $c'\not=1_{\hat Y}$.)
\end{lemma}

\begin{equation}\label{commutative-exact-diagram-stability-a}
\xymatrix@C=2cm{0&0\\
\hat X\ar[u]\ar[r]^{\hat f}&\hat Y\ar[u]\\
\hat X\ar[u]_{c}\ar[r]_{\hat f}&\hat Y\ar[u]^{c'}\\}
\end{equation}
\begin{proof}
In fact one has $\gamma=\beta\circ\alpha:\hat f\to\hat f\to 0$ in $\mathfrak{T}$.
\end{proof}

Let us now consider some lemmas relating above bordism groups with homotopy deformations.
\begin{lemma}\label{stability-lemma-c}
The maps $\tilde\zeta$ and $\hat\zeta$ are related by homotopic morphisms.
\end{lemma}
\begin{proof}
 In fact, let us consider inside the critical strip, say $\Xi$, another strip, say $\Xi_0\subset \Xi$, parallel to $\Xi$. Let $\phi_t:\mathbb{C}\to \mathbb{C}$ be an homotopy (flow) that deforms only $\Xi$ in such a way that its boundary remains fixed (as well as all $\mathbb{C}$ outside $\Xi$) but collapses $\Xi_0$ onto the critical line, say $\Xi_c\subset\Xi_0\subset\Xi$. If on the boundary of $\Xi_0$ there are four zeros of $\tilde\zeta$ outside $\Xi_c$, then all these zeros collapse on two points on $\Xi_c$. (These points do not necessitate to be zeros of $\tilde\zeta$. We call such points {\em ghost zeros}.) Then we can reproduce the process followed to build the covering $\hat b$, to identify a homotopy $\hat b_t$ such that all zeros and ghost zeros on $\Xi_c$ are projected on $\pm z_0$. In such a way we get the commutative and exact diagram (\ref{exact-commutative-diagram-homotopy-zeros})
\begin{equation}\label{exact-commutative-diagram-homotopy-zeros}
\xymatrix@C=1cm{
\mathbb{C}\ar@/_1pc/[rr]_{\hat b_t}\ar@/^4pc/[rrrrr]^{\hat b_1}\ar@/^2pc/[rrrr]^(0.8){\phi_1}\ar[d]_{\tilde \zeta}\ar[r]^{\phi_t}&\mathbb{C}\ar[r]^{\underline{\hat b}_t}&
  \hat S^1\ar[d]^{\hat\zeta_t}\ar[r]&\cdots\ar[r]&\mathbb{C}\ar[r]_{\underline{\hat b}_1}&\hat S^1\ar[d]^{\hat\zeta_1}\ar[r]&0\\
  \hat S^1\ar[d]\ar@/_2pc/[rrrrr]_{\hat b'_1}\ar[rr]^{\hat b'_t}&&\hat S^1\ar[d]\ar[r]&\cdots\ar[rr]&&
  \hat S^1\ar[d]\ar[r]&0\\
  0&&0&&&0&\\}
\end{equation}
that now works without excluding that $\tilde\zeta$ has zeros outside the critical line. In the commutative diagram (\ref{exact-commutative-diagram-homotopy-zeros}) we have denoted by $\underline{\hat b}_t:\mathbb{C}\to\hat S^1$, $t\in[0,1]$, the mappings that project on the two zeros $\pm z_0$ of $\hat\zeta$, all zeros of $\tilde\zeta$ on the critical line $\Xi_c$, and also the ghost zeros arrived on $\Xi_c$ at the homotopy time $t\in[0,1]$. The mappings $\underline{\hat b}_t$ work similarly to $\hat b:\mathbb{C}\to\hat S^1$, i.e., the zeros-Riemann-zeta-normalized-universal-covering of $\hat S^1$. (For abuse of notation we could denote them yet by $\hat b$. However, in order to avoid any possible confusion, we have preferred to adopt a different symbol, emphasizing its relation to the flow $\phi_t$ and the ghost zeros at the homotopy time $t$.) With this respect we shall call {\em $t$-ghost zeros} the ghost zeros that at the time $t$ arrive on $\Xi_c$. Set $\hat\beta_t=(\hat b_t,\hat b'_t)$. Furthermore, $\hat\zeta_t$  denotes the pointed meromorphic function $\hat\zeta_t:\hat S^1\to \hat S^1$ that can have zeros outside the critical line, and hence a pole at $\infty$ with the multiplicity equal to the numbers of such zeros. For example whether such number is four then the divisor of $\hat\zeta_t$, $\forall t\in[0,1]$, is given in (\ref{divisor-homotopy}).\footnote{Warning. The assumption that the number of zeros of $\zeta(s)$, outside the critical line, is finite is not restrictive. In fact in Appendix D we show how one arrives to the same conclusions by assuming infinite such a set. There it is given also an explicit characterization of the continuous, singular flow $\phi_t:\mathbb{C}\to\mathbb{C}$, $t\in[0,1]$.}
\begin{equation}\label{divisor-homotopy}
\left\{\begin{array}{ll}
  (\hat\zeta_t)=&1\cdot s_{1,t}+1\cdot s_{2,t}+1\cdot s_{3,t}+1\cdot s_{4,t}+1\cdot q_{z_0}+1\cdot q_{-z_0}-1\cdot q_{0}-1\cdot q_{1}\\
  &-4\cdot q_{\infty},\, t\in[0,1),\\
  (\hat\zeta)=&(\hat\zeta_1)=1\cdot q_{z_0}+1\cdot q_{-z_0}-1\cdot q_{0}-1\cdot q_{1}.\\
    \end{array}\right.
\end{equation}
where $s_{i,t}$, $i=1,2,3,4$, denote the zeros at $t$, different from $\pm z_0$, on $\hat S^1$. Note the exactness of the short sequence (\ref{short-exact-sequences}), for any $t\in[0,1]$.  In particular one has $\hat\zeta_1=\hat\zeta$.

\begin{equation}\label{short-exact-sequences}
\xymatrix{\tilde\zeta\ar[r]^{\hat\beta_t}&\hat\zeta_t\ar[r]&0\\}
\end{equation}

In particular we get the exactness of the sequence $\xymatrix{\tilde\zeta\ar[r]^{\hat\beta_1}&\hat\zeta\ar[r]&0\\}$.
By using Lemma \ref{relations-bordism-groups} we can see that $\hat\Omega_\bullet(\hat\zeta)\cong\hat\Omega_\bullet(\tilde\zeta)$.\footnote{Warning. The morphisms $\hat\beta_t$ identify $\hat\zeta$ as an homotopy deformation of $\tilde\zeta$. This should not be enough to state that $\hat\Omega_\bullet(\hat\zeta)$ is isomorphic to $\hat\Omega_\bullet(\tilde\zeta)$. In fact, R. E. Stong has shown with a counterexample that homotopic maps do not necessitate have isomorphic map bordism groups.\cite{STONG} However in the case of the maps $\tilde\zeta$ and $\hat\zeta$ this result does not apply thanks to the fibration structures of the involved maps. See Appendix E for a detailed proof of the isomorphism $\hat\Omega_\bullet(\hat\zeta)\cong\hat\Omega_\bullet(\tilde\zeta)$ existence.} Let us emphasize that, since $\hat\zeta$ is a closed quantum-complex map, it identifies a bordism in $\hat\Omega_\bullet(\hat\zeta)$ by means of the canonical morphism $1=(1,1):\hat\zeta\to\hat\zeta$, i.e., the commutative diagram (\ref{commutative-diagram-identity-morphism}).
\begin{equation}\label{commutative-diagram-identity-morphism}
   \xymatrix{\hat S^1\ar[r]^{\hat\zeta}&\hat S^1\\
   \hat S^1\ar[u]^{1}\ar[r]_{\hat\zeta}&\hat S^1\ar[u]_{1}\\}
\end{equation}

Then taking into account the isomorphism $\hat\Omega_\bullet(\hat\zeta)\cong\hat\Omega_\bullet(\tilde\zeta)$, it follows that $\hat\zeta$ identifies a bordism class in $\hat\Omega_\bullet(\tilde\zeta)$. Therefore, must exist a morphism $\delta=(\epsilon,\epsilon'):\hat\zeta\to\tilde\zeta$, namely must exist a commutative diagram (\ref{commutative-diagram-existence}).
\begin{equation}\label{commutative-diagram-existence}
   \xymatrix{\mathbb{C}\ar[r]^{\tilde\zeta}&\hat S^1\\
   \hat S^1\ar[u]^{\epsilon}\ar[r]_{\hat\zeta}&\hat S^1\ar[u]_{\epsilon'}\\}
\end{equation}
with $\epsilon\in Hom_{\mathfrak{T}}(\hat S^1,\mathbb{C})$ and $\epsilon'\in Hom_{\mathfrak{T}}(\hat S^1,\hat S^1)$.
From Lemma \ref{bordism-morphism-and-quantum-riemann-zeta-function} we know that such a morphism $\delta$ is just $\delta=(\epsilon,\epsilon')=(\pi_*,\pi'_*)=\alpha$. Thus we get a surjective morphism $\xymatrix{\hat\zeta\ar[r]^{\alpha}&\tilde\zeta\ar[r]&0\\}$. This composed with $\hat\beta_1$ allows us to state that $\hat\zeta$ is $\hat\beta_1=(\hat b_1,\hat b'_1)$-stable in $\hat\Omega_\bullet(\tilde\zeta)$, in the sense of Definition \ref{stability-b}. In fact one has the exact commutative diagram of maps and morphisms given in (\ref{exact-commutative-diagram-maps-morphisms-stability-a}), i.e., $\omega=\hat\beta_1\circ\alpha:\hat\zeta\to\hat\zeta$ is a non trivial surjective morphism.
\begin{equation}\label{exact-commutative-diagram-maps-morphisms-stability-a}
  \xymatrix{\hat\zeta\ar[dr]_{\omega}\ar[r]^{\alpha}&\tilde\zeta\ar[d]^{\hat\beta_1}\ar[r]&0\\
  &\hat\zeta\ar[d]&\\
  &0&\\}
\end{equation}
\end{proof}
\begin{definition}\label{stability-lemma-definition-a}
We say that a morphism $\alpha:\tilde\zeta\to\hat f$ is a {\em perturbation of $\tilde\zeta$ inside the bordism group} $\hat\Omega_\bullet(\tilde\zeta)$, if $\hat f$ identifies a bordism class of $\hat\Omega_\bullet(\tilde\zeta)$.\footnote{This implies that $\hat f$ is a closed map of $\mathfrak{Q}_{\mathbb{C}}$ and there is a morphism  of $\mathfrak{T}$, $\beta:\hat f\to \tilde\zeta$.}
\end{definition}
\begin{definition}\label{stability-lemma-definition-b}
We say that a morphism $\alpha:\tilde\zeta\to\hat f$ is a {\em perturbation of $\tilde\zeta$ outside the bordism group} $\hat\Omega_\bullet(\tilde\zeta)$, if $\hat f$ does not identify a bordism class of $\hat\Omega_\bullet(\tilde\zeta)$.
\end{definition}

\begin{lemma}\label{stability-lemma-d}
When one assumes that $\tilde\zeta$ has a finite set of zeros outside the critical line, the perturbations $\hat\beta_t:\tilde\zeta\to\hat\zeta_t$, $t\in[0,1]$, of $\tilde\zeta$ are inside $\hat\Omega_\bullet(\tilde\zeta)$. Furthermore, one has the canonical isomorphisms reported in {\em(\ref{isomorphisms-bordism-groups-related-toriemann-zeta-function})}.
\begin{equation}\label{isomorphisms-bordism-groups-related-toriemann-zeta-function}
    \hat\Omega_\bullet(\tilde\zeta)\cong\hat\Omega_\bullet(\hat\zeta)\cong\hat\Omega_\bullet(\hat\zeta_t),\, \forall t\in(0,1).
\end{equation}
One can identify the canonical bordism classes of $\hat\Omega_\bullet(\hat\zeta)$ and $\hat\Omega_\bullet(\hat\zeta_t)$ respectively, namely one has the one-to-one correspondence in the above isomorphism {\em(\ref{isomorphisms-bordism-groups-related-toriemann-zeta-function})}, $[\hat\zeta]\leftrightarrow[\hat\zeta_t]$, $\forall t\in(0,1]$.

Furthermore, the maps $\hat\zeta_t$, $\forall t\in(0,1]$, are $\hat\beta_t$-stable in $\hat\Omega_\bullet(\tilde\zeta)$.
\end{lemma}
\begin{proof}
Let us assume that $\tilde\zeta$ has a finite set of zero outside the critical line. Then the perturbations $\hat\beta_t:\tilde\zeta\to\hat \zeta_t$ of $\tilde\zeta$, considered in the commutative diagram (\ref{exact-commutative-diagram-homotopy-zeros}), are inside $\hat\Omega_\bullet(\tilde\zeta)$. In fact, each map $\hat\zeta_t:\hat S^1\to\hat S^1$ is a closed map in $\mathfrak{Q}_{\mathbb{C}}$, hence it identifies a bordism class in $\hat\Omega_\bullet(\tilde\zeta_t)$. Furthermore, by following a similar road used to prove the isomorphism considered in Appendix E, we can prove that the homomorphism $\hat\Omega_\bullet(\hat\beta_t):\hat\Omega_\bullet(\tilde\zeta)\to \hat\Omega_\bullet(\hat\zeta_t)$ is an isomorphism. Therefore, $[\hat\zeta_t]\in \hat\Omega_\bullet(\hat\zeta_t)$ can be identified by an equivalence class of $\hat\Omega_\bullet(\tilde\zeta)$. Thus it must exist a morphism $\hat\gamma_t=(e_t,e'_t):\hat\zeta_t\to\tilde\zeta$ such that diagram (\ref{commutative-diagram-perturbed-morphisms-and-bordisms}) is commutative.
\begin{equation}\label{commutative-diagram-perturbed-morphisms-and-bordisms}
    \xymatrix{\mathbb{C}\ar[r]^{\tilde\zeta}&\hat S^1\\
    \hat S^1\ar[u]^{e_t}\ar[r]_{\hat\zeta_t}&\hat S^1\ar[u]_{e'_t}\\}
\end{equation}
Really we can take $\hat\gamma_t=\alpha\circ\hat\beta_{t,1}$, where $\hat\beta_{t,1}:\hat\zeta_t\to\hat\zeta_1=\hat\zeta$ is identified in diagram (\ref{exact-commutative-diagram-homotopy-zeros}) and $\alpha=(\pi_*,\pi'_*):\hat\zeta\to\tilde\zeta$. Since all the maps involved are surjective, one has the exact commutative diagram (\ref{exact-commutative-diagram-maps-morphisms-stability-b}), with $\nu=\hat\beta_t\circ\hat\gamma_t$.
\begin{equation}\label{exact-commutative-diagram-maps-morphisms-stability-b}
    \xymatrix{\hat\zeta_t\ar[dr]_{\nu}\ar[r]^{\hat\gamma_t}&\tilde\zeta\ar[d]^{\hat\beta_t}\ar[r]&0\\
    &\hat\zeta_t\ar[d]&\\
    &0&}
\end{equation}
Therefore, $\hat\zeta_t$ is $\hat\beta_t$-stable in $\hat\Omega_\bullet(\tilde\zeta)$.
Finally, by considering the isomorphism $\hat\Omega_\bullet(\hat\zeta)\cong \hat\Omega_\bullet(\hat\zeta_t)$, one can identify $[\hat\zeta]\in\hat\Omega_\bullet(\hat\zeta)$ with $[\hat\zeta_t]\in\hat\Omega_\bullet(\hat\zeta_t)$. In fact, by following a proceeding  similar to one considered in Appendix E, we can see that if $\hat f:M\to N$ is a closed map in $\mathfrak{Q}_{\mathbb{C}}$, equivalent to $\hat\zeta_t$ in $\hat\Omega_\bullet(\hat\zeta_t)$, then it is equivalent to $\hat\zeta$ in $\hat\Omega_\bullet(\hat\zeta)$ too.
\end{proof}
\begin{lemma}\label{stability-lemma-e}
 When one assumes that $\tilde\zeta$ has an infinite set of zeros outside the critical line, (see Appendix E), the perturbations $\tilde\beta_t:\tilde\zeta\to\tilde\zeta_t$, $\forall t\in(0,1)$, of $\tilde\zeta$ are outside $\hat\Omega_\bullet(\tilde\zeta)$.
\end{lemma}
\begin{proof}
The perturbations $\tilde\beta_t:\tilde\zeta\to\tilde\zeta_t$ considered in Appendix E, when one assumes that $\tilde\zeta$ has an infinite set of zeros outside the critical line, are outside the bordism group $\hat\Omega_\bullet(\tilde\zeta)$, since the maps $\tilde\zeta_t:\mathbb{C}\to\mathbb{C}$ are non-compact and non-closed, hence they cannot identify bordism equivalence classes in $\hat\Omega_\bullet(\tilde\zeta)$.
\end{proof}

\begin{lemma}\label{stability-lemma-b}
To state that the Riemann hypothesis is true is equivalent to say that the quantum-complex zeta Riemann function $\hat\zeta$ is $\hat\beta$-stable in the bordism group $\hat\Omega_\bullet(\tilde\zeta)$, where $\hat\beta=(\hat b,\hat b'):\tilde\zeta\to\hat\zeta$ is a perturbation of $\tilde\zeta$ inside $\hat\Omega_\bullet(\tilde\zeta)$.
\end{lemma}
\begin{proof}
The quantum-complex zeta Riemann function $\hat\zeta$ is $\hat\beta=(\hat b,\hat b')$-stable in $\hat\Omega_\bullet(\tilde\zeta)$ iff $\tilde\zeta$ has not zeros outside the critical line. In fact one has $\gamma=\hat\beta\circ\alpha:\hat \zeta\to\hat \zeta\to 0$ with $\alpha=(\pi_*,\pi'_\star)$ (see diagram (\ref{commutative-diagram-morphism-completed-riemann-zeta-function})) and $\hat\beta=(\hat b,\hat b')$ (see diagram (\ref{commutative-diagram-quantum-zeta-riemann-function-a})). But $\gamma$ is surjective iff $\hat\beta$ is surjective, i.e., iff $\tilde\zeta$ has  not zeros outside the critical line (see Remark \ref{remark-zeros-outside-critical-line}).
Therefore to state that the Riemann hypothesis is true is equivalent to say that the quantum-complex zeta Riemann function $\hat\zeta$ is $\hat\beta$-stable in the bordism group $\hat\Omega_\bullet(\tilde\zeta)$ in the sense of Definition \ref{stability-b}, hence the perturbation $\hat\beta$ is inside $\hat\Omega_\bullet(\tilde\zeta)$, (see Lemma \ref{bordism-morphism-and-quantum-riemann-zeta-function} and Definition \ref{stability-lemma-definition-a}).
\end{proof}

\begin{definition}\label{definition-set-ghost-zeros}
$\bullet$\hskip 2pt Let us denote by $Z_c(\tilde\zeta)$ the set of zeros of $\tilde\zeta$ in the critical line. We call $Z_c(\tilde\zeta)$ the {\em set of critical zeros} of $\tilde\zeta$.

$\bullet$\hskip 2pt Let us denote by $Z_0(\tilde\zeta)$ the set of zeros of $\tilde\zeta$ that are outside the critical line. Let us call $Z_0(\tilde\zeta)$ the {\em set of outside-ghost zeros} of $\tilde\zeta$.

$\bullet$\hskip 2pt We denote by $Z_{c,t}(\tilde\zeta)$ the {\em set of $t$-ghost zeros} of $\tilde\zeta$, i.e., the set of ghost zeros that at the homotopy time $t\in[0,1]$ are on the critical line $\Xi_c$. (For $t=0$, $Z_{c,t}(\tilde\zeta)$ is empty: $Z_{c,0}(\tilde\zeta)=\varnothing$.)
\end{definition}

In Tab. \ref{alternative-possible-properties-of-non-critical-zeros-of-completed-riemann-zeta-function} are resumed the possible characterizations of $Z_0(\tilde\zeta)$.\footnote{There propositions \textbf{(a)}, \textbf{(b)} and \textbf{(c)} are two by two contradictory propositions. But the law of excluded middle does not work for them. This is instead an example of $3$-valued logics and to these it applies the law of excluded $4$-th.}

\begin{table}[t]
\caption{Possible and alternative propositions on the set $Z_0(\tilde\zeta)$ of outside-ghost zeros of $\tilde\zeta$.}
\label{alternative-possible-properties-of-non-critical-zeros-of-completed-riemann-zeta-function}
\begin{tabular}{|c|l|l|l|}
\hline
\hfil{\rm{\footnotesize Quotation}}\hfil&\hfil{\rm{\footnotesize Proposition}}\hfil&\hfil{\rm{\footnotesize $\hat\Omega_\bullet(\hat\zeta)$}}\hfil&\hfil{\rm{\footnotesize Stability of $\hat\zeta$ in $\hat\Omega_\bullet(\tilde\zeta)$}}\hfil\\
\hline
\hfil{\rm{\footnotesize $\mathbf{(a)}$}}\hfil&\hfil{\rm{\footnotesize $Z_0(\tilde\zeta)=\varnothing$}}\hfill&\hfil{\rm{\footnotesize $\hat\Omega_\bullet(\hat\zeta)\cong\hat\Omega_\bullet(\tilde\zeta)$}}\hfil&\hfil{\rm{\footnotesize $\hat\beta$-stable, $\hat\beta=(\hat b,\hat b')$, (see diagram (\ref{commutative-diagram-quantum-zeta-riemann-function-a}).}}\hfill\\
\hline
\hfil{\rm{\footnotesize $\mathbf{(b)}$}}\hfil&\hfil{\rm{\footnotesize $\sharp(Z_0(\tilde\zeta))<\aleph_0$}}\hfill&\hfil{\rm{\footnotesize $\hat\Omega_\bullet(\hat\zeta)\cong\hat\Omega_\bullet(\tilde\zeta)$}}\hfil&\hfil{\rm{\footnotesize $\hat\beta_1$-stable, $\hat\beta_1=(\hat b_1,\hat b'_1)$, (see diagram (\ref{exact-commutative-diagram-homotopy-zeros}).}}\hfill\\
\hline
\hfil{\rm{\footnotesize $\mathbf{(c)}$}}\hfil&\hfil{\rm{\footnotesize $\sharp(Z_0(\tilde\zeta))=\aleph_0$}}\hfill&\hfil{\rm{\footnotesize $\hat\Omega_\bullet(\hat\zeta)\cong\hat\Omega_\bullet(\tilde\zeta)$}}\hfil&\hfil{\rm{\footnotesize $\hat\beta_1$-stable, $\hat\beta_1=(\hat b_1,\hat b'_1)$, (see diagram (\ref{appendix-d-exact-commutative-diagram-infinite-divisor-guess-c}).}}\hfill\\
\hline
\multicolumn{4}{l}{\rm{\footnotesize $Z_0(\tilde\zeta)$: See Definition \ref{definition-set-ghost-zeros}.}}\hfill\\
\end{tabular}
\end{table}
From the proof of the isomorphism $\hat\Omega_\bullet(\tilde\zeta)\cong \hat\Omega_\bullet(\hat\zeta)$  and from Lemma \ref{relations-bordism-groups} it follows that also $\hat\Omega_\bullet(\hat\beta)$ is an isomorphism. In fact, with respect to the quotation reported in Tab. \ref{alternative-possible-properties-of-non-critical-zeros-of-completed-riemann-zeta-function}, let us consider first the case \textbf{(b)},\footnote{The conclusion in the case \textbf{a} is obvious.} i.e., let us assume that there is a finite set of zeros of $\tilde\zeta$, outside the critical line. Then the homotopy relation between $\tilde\zeta$ and $\hat\zeta$ allows us to relate these last maps with the commutative and exact diagram (\ref{commutative-diagram-identification})(A) with the continuous maps $\hat b_1:\mathbb{C}\to \hat S^1$, $\hat b'_1:\hat S^1\to\hat S^1$ (say the {\em dot maps}). (Compare with the commutative diagram (\ref{exact-commutative-diagram-homotopy-zeros}).) This diagram allows us to understand that $\hat\zeta$ is $\hat\beta_1$-stable. Let us now consider the case \textbf{(c)}, i.e., let us assume that there is an infinite set of zeros of $\tilde\zeta$, outside the critical line. Then from the commutative diagram (\ref{appendix-d-exact-commutative-diagram-infinite-divisor-guess-c}) one can build the commutative diagram (\ref{commutative-diagram-identification})(B)
with the continuous maps $\hat b_1:\mathbb{C}\to\hat S^1$ and $\hat b'_1:\hat S^1\to\hat S^1$ (say again the {\em dot maps}). This diagram allows us to understand again that $\hat\zeta$ is $\hat\beta_1=(\hat b_1,\hat b'_1)$-stable.
Therefore, the unique obstruction to the condition $Z_0(\tilde\zeta)=\varnothing$ is the absence of $\hat\beta=(\hat b,\hat b')$-stability for $\hat\zeta$ in the bordism group $\hat\Omega_\bullet(\tilde\zeta)$.

\begin{equation}\label{commutative-diagram-identification}
    \, {\rm(A)}\, \xymatrix{&\mathbb{C}\ar[dr]_(0.3){\underline{\hat b}_1}&&\\
    \mathbb{C}\ar@/^4pc/[rr]^{\hat b_1}\ar[ur]_(0.7){\phi_1}\ar@{.>}[r]^{\hat b_1}\ar[d]_{\tilde\zeta}&\hat S^1\ar@{=}[r]\ar[d]_{\hat\zeta}&\hat S^1\ar[d]^{\hat\zeta_1}\ar[r]&0\\
    \hat S^1\ar@/_4pc/[rr]_{\hat b'_1}\ar[d]\ar@{.>}[r]_{\hat b'_1}&\hat S^1\ar[d]\ar@{=}[r]&\hat S^1\ar[d]\ar[r]&0\\
    0&0&0&\\}
    \, {\rm(B)}\,
     \scalebox{0.9}{$\xymatrix{&&&&\mathbb{C}\ar[dddd]^{\tilde\zeta_1}\ar[dll]_{\underline{\hat b}_1}\ar[r]&0\\
     \mathbb{C}\ar@/^3pc/[urrrr]^{\phi_1}\ar@{.>}[rr]^{\hat b_1}\ar[dd]_{\tilde\zeta}&&\hat S^1\ar[dl]\ar[dd]^{\hat\zeta}\ar[r]&0&&\\
     &0&&&&\\
    \hat S^1\ar@/_3pc/[drrrr]_{\phi'_1}\ar[d]\ar@{.>}[rr]_{\hat b'_1}&&\hat S^1\ar[lu]\ar[d]\ar[r]&0&&\\
    0&&0&&\hat S^1\ar[ull]^{\underline{\hat b}'_1}\ar[d]\ar[r]&0\\
    &&&&0&\\}$}
     \end{equation}
On the other hand since the isomorphism $\hat\Omega_\bullet(\hat\zeta)\cong\hat\Omega_\bullet(\tilde\zeta)$ holds independently on which of the propositions \textbf{(a)}, \textbf{(b)} and \textbf{(c)}  in Tab. \ref{alternative-possible-properties-of-non-critical-zeros-of-completed-riemann-zeta-function} is true, we can assume such isomorphism without referring to a specific of such propositions. Then, since the continuous mapping $\hat b:\mathbb{C}\to\hat S^1$ is canonical and independent from above propositions, we can state that to any map $\hat f\in Hom_{\mathfrak{Q}_{\mathbb{C}}}(M,N)$, identifying a bordism element in $\hat\Omega_\bullet(\tilde\zeta)$, we must be able to associate by means of the canonical mapping $\hat b$, a bordism element of $\hat\Omega_\bullet(\hat\zeta)$. In other words, we must be able to associate to $\hat f$ and $\hat b$ a morphism $\hat\beta=(\hat b,\hat b')$ sending $\hat f$ to $\hat\zeta$. This means that we can associate to $\hat b$ a continuous map $\hat b'\in Hom_{\mathfrak{T}}(\hat S^1,\hat S^1)$ such that diagram (\ref{concluding-commutative-diagram}) is commutative (and exact).
But this just means that it must exist also the isomorphism $\hat\beta_\bullet:\hat\Omega_\bullet(\tilde\zeta)\cong\hat\Omega_\bullet(\hat\zeta)$, (according to Lemma \ref{relations-bordism-groups}), and that $\hat\zeta$ is $\hat\beta$-stable. Therefore from Lemma \ref{stability-lemma-b} we get that proposition \textbf{(a)}, $Z_0(\tilde\zeta)=\varnothing$, in Tab. \ref{alternative-possible-properties-of-non-critical-zeros-of-completed-riemann-zeta-function}, is the unique true proposition. Then from Lemma \ref{complted-riemann-zeta-function} it follows that also the Riemann zeta function $\zeta$ cannot have non-trivial zeros outside the critical line. The proof of Theorem \ref{main-proof} is now complete.
\begin{equation}\label{concluding-commutative-diagram}
   \scalebox{0.8}{$\xymatrix{M\ar[ddd]_{\hat f}\ar[rd]^{a}\ar@/^2pc/[drr]^{b}&&&\\
    &\mathbb{C}\ar[r]^{\hat b}\ar[d]_{\tilde\zeta}&\hat S^1\ar[d]^{\hat\zeta}\ar[r]&0\\
   &\hat S^1\ar@{.>}[r]^{\hat b'}\ar[d]&\hat S^1\ar[d]\ar[r]&0\\
    N\ar@/_2pc/[urr]_{b'}\ar[ur]_{a'}&0&0&\\}$}
\end{equation}
\end{proof}

\begin{appendices}
\appendix{\bf Appendix A: Proof that $\mathbf{\hat\zeta(s):\hat S^1\to\hat S^1}$ is a meromorphic function with two zeros and two simple poles at finite.}\label{appendixa}
\renewcommand{\theequation}{A.\arabic{equation}}
\setcounter{equation}{0}  % reset counter

In this appendix we shall give an explicit constructive proof that the meromorphic completed Riemann function $\tilde\zeta:\mathbb{C}\to\mathbb{C}$ identifies a meromorphic function $\hat\zeta:\hat S^1\to\hat S^1$ having two zeros and two simple poles at finite.

Let us denote by $Z(\tilde\zeta)\subset \mathbb{C}$ and $Pol(\tilde\zeta)\subset\mathbb{C}$ respectively the set of zeros and poles of $\tilde\zeta$. $\tilde\zeta$ identifies a canonical mapping, yet denoted by $\tilde\zeta:\mathbb{C}\to\hat S^1=\mathbb{C}\bigcup\{\infty\}$, that is a $\mathbb{C}$-quantum mapping, (i.e., a holomorphic mapping) relative to the structure of $\mathbb{C}$-quantum manifold of both $\mathbb{C}$ and $\hat S^1$. One has $Pol(\tilde\zeta)=\tilde\zeta^{-1}(\infty)$, and $Z(\tilde\zeta)=\tilde\zeta^{-1}(0)$, if $0\in\mathbb{C}$. Let us denote by $(\tilde\zeta)_p$ the degree of $\tilde\zeta$ at the point $p\in \mathbb{C}$. Then one has
\begin{equation}\label{appendix-a-degree-at-point}
(\tilde\zeta)_p=\left\{\begin{array}{ll}
                          0&,\, p\not\in Z(\tilde\zeta)\bigcup Pol(\tilde\zeta)\\
                          +1&,\, p\in Z(\tilde\zeta)\\
                          -1&,\, p\in Pol(\tilde\zeta)\
                        \end{array}\right.
\end{equation}

We shall consider, now, some lemmas.
\begin{alemma}[Unique presentation by poles]\label{appendix-a-lemma-a}
Let $f:\hat S^1\to\hat S^1$ be a meromorphic mapping, i.e., $f$ is not a constant map $f(z)=\infty$, $\forall z\in \hat S^1$. Then $f$ is uniquely expressible as
\begin{equation}\label{appendix-a-meromorphic-presentation-poles}
    f(z)=p(z)+\sum_{i,j}\frac{c_{ij}}{(z-p_i)^j}
\end{equation}
where $p(z)$ is a polynomial, the $c_{ij}$ are constants and $p_i\in Pol(f|_{\mathbb{C}})$, i.e., $p_i$ are the finite poles of $f$. Since $\sharp(Pol(f))<\aleph_0$, the above sum is finite.
\end{alemma}
\begin{proof}
Even if this and some other lemma in this proof are standard, we will sketch a proof to better clarify the proof of Lemma~\ref{quantum-complex-zeta-riemann-function} and Theorem~\ref{main-proof}. Since $f$ is meromorphic, it admits a convergent Laurent expression near a pole $p$.
\begin{equation}\label{appendix-a-laurent-expression-near-a-pole}
    \begin{array}{l}
    a_{-n}(z-p)^{-n}+a_{-n+1}(z-p)^{-n+1}+\cdots+ a_{-1}(z-p)^{-1}+\sum_{k\ge 0}a_k(z-p)^k \\
    =f_p(z)+\sum_{k\ge 0}a_k(z-p)^k. \\
  \end{array}
\end{equation}

The negative powers form the {\em principal part} $f_p(z)$ of the series. The poles are isolated and the compactness of $\hat S^1$ implies that $\sharp(Pol(f))<\aleph_0$. Set $\psi(z)=f(z)-\sum_{p_i\in Pol(f)} f_{p_i}$. Then $\psi(z)$ is a meromorphic function with poles only at $\infty$. Let us consider the representation of $\psi(z)$ in the open neighborhood of $\infty\in\hat S^1$, i.e., in the standard chart $(\hat S^1\setminus\{0\},\varphi(z)=\frac{1}{w})$. We get
\begin{equation}\label{appendix-a-representation-of-psi}
\begin{array}{ll}
\psi(w)&=a_{-n}w^{-n}+a_{-n+1}w^{-n+1}+\cdots+ a_{-1}w^{-1}+\sum_{k\ge 0}a_kw^k\\
&=\psi_0(w)+\sum_{k\ge 0}a_kw^k.\\
 \end{array}
\end{equation}

Then $\psi(w)-\psi_0(w)=\sum_{k\ge 0}a_kw^k$ is a holomorphic function without poles on $\hat S^1$. We can use the following lemmas.
\begin{alemma}\label{appendix-a-lemma-b}
Every holomorphic function defined everywhere on a compact connected Riemann surface is constant.
\end{alemma}
\begin{proof}
This is standard.
\end{proof}
Therefore we can conclude that $\psi(w)-\psi_0(w)=C\in\mathbb{C}$.  By conclusion $\psi(z)$ is a polynomial in the neighborhood of $\infty$: $\psi(z)=a_{-n}z^{n}+a_{-n+1}z^{n-1}+\cdots+ a_{-1}z+C$.
\end{proof}

\begin{alemma}[Unique presentation by zeros and poles]\label{appendix-a-lemma-c}
A meromorphic function $f:\hat S^1\to\hat S^1$ has a unique expression
\begin{equation}\label{appendix-a-equation-appendix-lemma-c}
f(z)=c\, \frac{\prod_{1\le i\le n}(z-z_i)^{h_i}}{\prod_{1\le j\le m}(z-p_j)^{k_j}},\, c\in \mathbb{C},\, z_i\in Z(f),\, p_j\in Pol(f).
\end{equation}

Here $h_i$ and $k_j$ are the multiplicities of zero and poles at finite.
Furthermore one has $\sharp(Z(f))=\sharp(Pol(f))$.
\end{alemma}
\begin{proof}
The function
\begin{equation}\label{appendix-a-meromorphic-function-a}
    g(z)=f(z)\diagup\left[\frac{\prod_{1\le i\le n}(z-z_i)^{h_i}}{\prod_{1\le j\le m}(z-p_j)^{k_j}}\right]
\end{equation}

is a meromorphic function on $\hat S^1$ without zeros and poles in $\mathbb{C}$. But a polynomial without roots in $\mathbb{C}$ is a constant $c\in\mathbb{C}$, hence we get the expression (\ref{appendix-a-equation-appendix-lemma-c}). From (\ref{appendix-a-equation-appendix-lemma-c}) we can obtain the behaviour of $f(z)$ in the standard chart
$(\hat S^1\setminus\{0\},\varphi(z)=\frac{1}{w})$:
\begin{equation}\label{appendix-a-behaviour-of-f-in-the-standard-chart}
f(\frac{1}{w})=c\, \frac{\prod_{1\le i\le n}(w^{-1}-z_i)}{\prod_{1\le j\le m}(w^{-1}-p_j)^{k_j}}=c\, w^{m-n}\frac{\prod_{1\le i\le n}(1-wz_i)}{\prod_{1\le j\le m}(1-wp_j)^{k_j}}
\end{equation}

hence
\begin{equation}\label{appendix-a-limits-of-f-in-the-standard-chart}
\mathop{\lim}\limits_{z\to\infty}f(z)=\mathop{\lim}\limits_{z\to\infty}z^{n-m}=\left\{\begin{array}{ll}
                                                                                                \infty& (n-m)>0,\, \hbox{\rm pole of order $n-m$}\\
                                                                                                0& (n-m)<0,\, \hbox{\rm zero of order $m-n$}.
                                                                                              \end{array}\right.
\end{equation}

Therefore if the number $n$ of finite zeros and the number $m$ of finite poles are such that $n-m>0$ (resp. $n-m<0$) $f:\hat S^1\to \hat S^1$ has at $\infty$, a pole (resp. a zero) of order $n-m$ (resp. $m-n$).  In the first case (resp. second case) we get that the number of poles is $m+(n-m)=n$ equal to the number of zeros (resp. the number of zero is $n+(m-n)=m$ equal to the number of poles)
\end{proof}
We have also the following well-known lemmas.
\begin{alemma}\label{appendi-a-lemma-d}
Two meromorphic functions on a compact Riemann surfaces having the same principal part at each of their poles must differ by a constant.
\end{alemma}

\begin{alemma}\label{appendix-a-lemma-e}
Two meromorphic functions on a compact Riemann surfaces having the same poles and zeros (multiplicities included) must differ by a constant factor.
\end{alemma}
 \begin{figure}[t]
 \includegraphics[height=4cm]{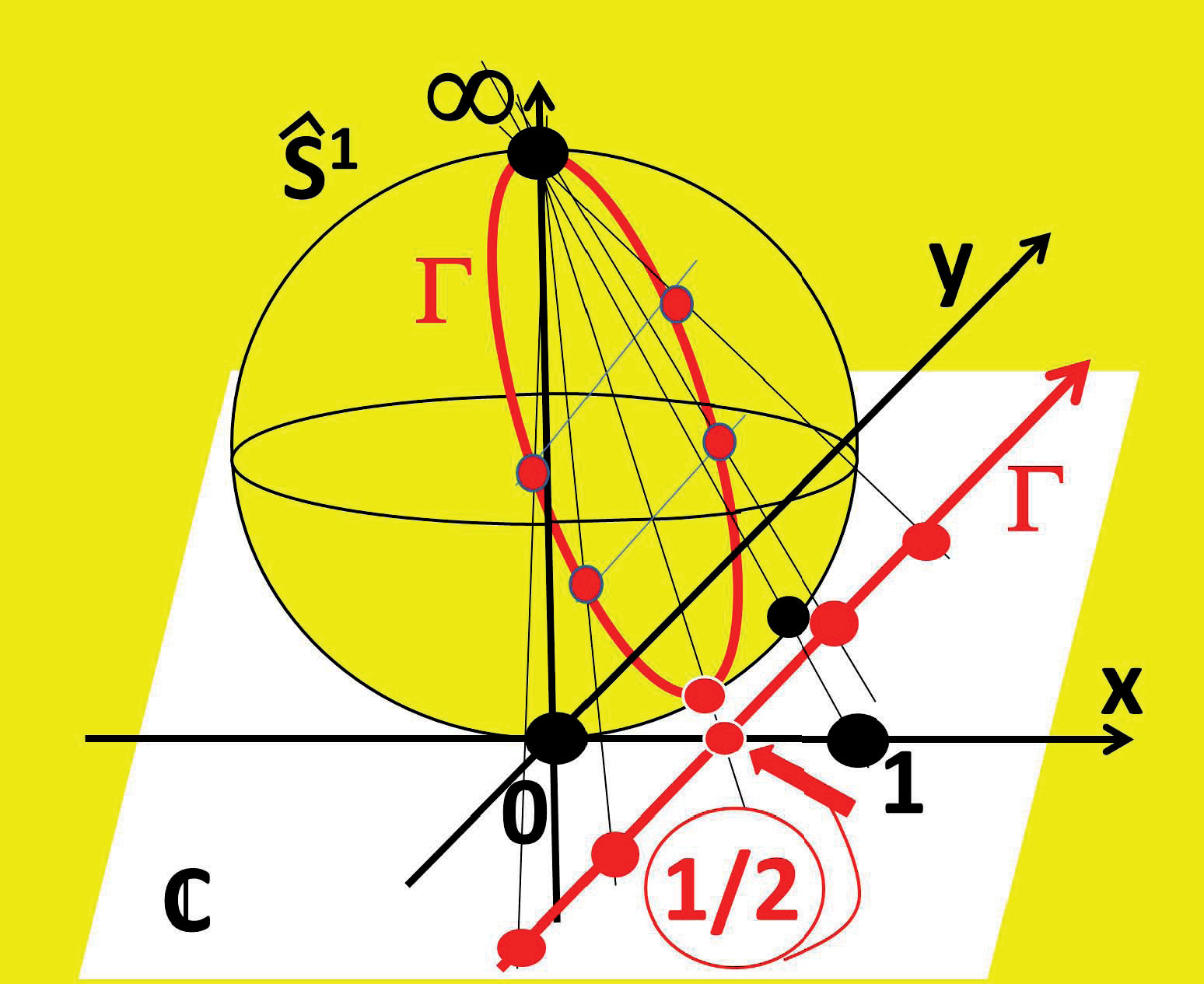}
\renewcommand{\figurename}{Fig.}
\caption{Representation of the (red) critical circle $\Gamma_{\frac{1}{2}}=\Gamma\subset\hat S^1$ by means of the stereographic projection of the (red) critical line $s=\frac{1}{2}$ (again denoted by $\Gamma$), in the complex plane $\mathbb{C}\cong\mathbb{R}^2$. There are also represented the first two zeros of $\hat\zeta(s)$ on $\Gamma$, besides some other symmetric couples of zeros identified by $\tilde\zeta(s)$. (Here $x=\Re(s)$ and $y=\Im(s)$.)  The stereographic projection identifies also the (red) point $\frac{1}{2}$ of the $x$-axis, with a (red) point on the circle in $\hat S^1$, representing the compactified $x$-axis, that is in the middle of the arc between $0$ and the (black) point corresponding to $1$ in the stereographic projection. Such black points represent the poles of $\hat\zeta:\hat S^1\to\hat S^1$.}
\label{appendix-a-critical-circle-and-symmetric-zero-couples}
\end{figure}
Let us assume, now, that $\hat\zeta:\hat S^1\to\hat S^1$ is a meromorphic function having $n$ simple zeros at finite coinciding with the ones of $\tilde \zeta$ that are nearest to the $x$-axis ($\Im(s)=0$) in the $\mathbb{C}$-plane. Let us assume that the poles of $\hat\zeta$ at finite are in correspondence one-to-one with the poles $\{0,1\}$ of $\tilde\zeta$, by means of the stereographic projection $\hat S^1\to\mathbb{C}$. By resuming we assume that
$\sharp(Z_{finite}(\hat\zeta))=\sharp(\{z_1,\cdots,z_n\})=n>\sharp(Pol_{finite}(\hat\zeta))=2=m$. Then from above results we get that $\hat\zeta$ is uniquely identified, up to a multiplicative constant $c$, by the following expression:
\begin{equation}\label{appendix-a-quantum-riemann-zeta-expression}
    \hat\zeta(s)=c\, \frac{\prod_{1\le i\le n}(s-z_i)}{s(s-1)}.
\end{equation}

We know also that such a meromorphic function $\hat\zeta:\hat S^1\to\hat S^1$, has a pole of order $n-2$ at $\infty$. Furthermore we can fix the constant $c$ by imposing that at some finite point $s_0\in\mathbb{C}$, one has $\hat\zeta(s_0)=\tilde\zeta(s_0)$. We can, for example take $s_0=\frac{1}{2}$. Then we have:\footnote{We have used the following valuations: $\pi^{-\frac{1}{4}}\thickapprox 102.87\times 10^{-4}$, $\Gamma(\frac{1}{4})\thickapprox 3.62$, $\zeta(\frac{1}{2})\thickapprox -1.4603$.}
\begin{equation}\label{appendix-a-constant-estimate}
c=-\frac{1}{4}\frac{\tilde\zeta(\frac{1}{2})}{\prod_{1\le i\le n}(\frac{1}{2}-z_i)},\, \tilde\zeta(\frac{1}{2})=\pi^{-\frac{1}{4}}\Gamma(\frac{1}{4})\zeta(\frac{1}{2})\thickapprox-0.05438.
\end{equation}
(Note that the set of zeros of $\tilde\zeta$ is not finite, since $\tilde\zeta$ is a meromorphic function on the noncompact Riemann surface $\mathbb{C}$.\footnote{A meromorphic function $f:X\to\mathbb{C}$, can have $\sharp(Z(f))=\aleph_0$. For example $\sharp(Z(\zeta))=\aleph_0$, since its trivial zeros are numbered by even negative integers, and zeros over the critical line were proved to be infinite by G. H. Hardy \cite{HARDY}. As a by product it follows that $\sharp(Z(\tilde\zeta))=\aleph_0$.}) The meromorphic function $\hat\zeta:\hat S^1\to\hat S^1$ satisfies the condition $n-(2+n-2)=0=\hbox{\rm deg}(\hat\zeta)$.
Therefore, from above calculations it should appear that $\hat\zeta(s)$ has a pole of order $n-2$ at $\infty$. On the other hand we can consider $\hat S^1=D_-\bigcup D_+$, with $D_+$ (resp. $D_-$) a little disk centered at $\infty$ (resp. $D_-$ a disk centered at $0$) and with $\Xi=\partial D_+=\partial D_-\thickapprox S^1$. We can assume that $\hat\zeta(s)$ has not zeros in $D_+$. In fact since $\sharp(Z(\hat\zeta))<\aleph_0$, and we can exclude that $\infty\in Z(\hat\zeta)$, we can draw a little circle $\Xi$ around $\infty\in \hat S^1$, such that $\Xi=\partial D_+$, and such in $D_+$ do not fall zeros of $\hat\zeta(s)$. Then, in order that $\hat\zeta$ should be a quantum-extension of $\hat\zeta$ we shall require that $\hat\zeta|_{D_{-}}$ is a meromorphic function on $D_{-}$, hence must be ${\rm deg}(\hat\zeta|_{D_-})=n-2=0$, namely $n=2$. As a by product we get that $\infty$ is a pole of degree $0$, namely it is a regular value since $\mathop{\lim}\limits_{w\to 0}w^{0}=1$. Therefore, $\hat\zeta$ in order to be a quantum-extension of $\tilde\zeta$, must be a meromorphic function $\hat\zeta:\hat S^1\to\hat S^1$, with finite poles $0$ and $1$, and only two finite zeros belonging to the critical circle in $\hat S^1$. This means that in order to identify such a meromorphic function we can choice the two symmetric zeros on the critical circle exactly corresponding to the two symmetric zeros on the critical line of $\tilde\zeta$, that are nearest to the $x$-axis ($x=\Im(s)=0$). We call such mapping the {\em quantum-complex zeta Riemann function}.\footnote{Warn ! The quantum-complex zeta Riemann function is a regular function at $\infty\in\hat S^1$. More precisely one has $\hat\zeta(\infty)=c\approx 6.8046$. This can be directly obtained by considering that the two non trivial zeros of $\zeta(s)$, nearest to the $x$-axis, are $z_1\approx\frac{1}{2}+i\, 14.1347$ and $z_2\approx\frac{1}{2}-i\, 14.1347$. Then from (\ref{appendix-a-constant-estimate}) we get $c\approx-\frac{1}{4}\frac{-0.05438}{(14.1347)^2}\approx 6.8046$. Let us emphasize that $\hat\zeta|_{D_{+}}$ has not zero and poles, hence is a holomorphic function on the compact Riemann surface $D_{+}$. Then the maximum modulus principle (resp. minimum modulus principle) assures that either $\hat\zeta|_{D_{+}}:D_{+}\to \mathbb{C}$ is a constant function, or, for any point $s_0$ inside $D_{+}$ there exists other points arbitrarily close to $s_0$ at which $|\hat\zeta|_{D_{+}}|$ takes larger (resp. lower) values. So the first hypothesis should imply that $\hat\zeta|_{D_{+}}=c=\hat\zeta(\infty)$. In such a case $\hat\zeta$ should have all points in $D_{+}$ as critical points. But this contradicts the fact that $\hat\zeta$ is a meromorphic function with an unique finite critical value at $s=\frac{1}{2}$. Therefore, it remains to consider that $\hat\zeta|_{D_{+}}$ is not constant and that the maximum (resp. minimum) of $|\hat\zeta|_{D_{+}}|$ cannot be inside $D_{+}$ but on its boundary $\partial D_{+}$.}

\begin{aremark}
From the above proof it follows that the relation between the two meromorphic mappings $\tilde\zeta:\mathbb{C}\to\hat S^1$ and $\hat\zeta:\hat S^1\to\hat S^1$ is a morphism in the category $\mathfrak{T}$ between morphisms $\tilde\zeta,\, \hat\zeta\in Hom(\mathfrak{Q}_{\mathbb{C}})$. The situation is resumed in the commutative diagram~{\rm(\ref{commutative-diagram-quantum-zeta-riemann-function-a})}. In Fig.~\ref{appendix-a-critical-circle-and-symmetric-zero-couples} is represented the critical circle $\Gamma\subset\hat S^1$ by means of the stereographic projection of the critical line $s=\frac{1}{2}$
in the complex plane $\mathbb{C}\cong\mathbb{R}^2$.\footnote{Let us recall that the stereographic projection
$\pi_\star:\hat S^1\to\mathbb{C}$, is represented in complex coordinates $(z,w)\in\mathbb{C}^2$, by the function
$\pi_\star(z_0,w_0)=z=\frac{2z_0}{2-\Re(w_0)}\in\mathbb{C}$, such that $(z_0,w_0)\in \hat S^1\subset\mathbb{C}^2$,
hence satisfies the following complex algebraic equation $P(z_0,w_0)=z_0\bar z_0-1+[\Re(w_0)-1]^2=0$.
[One can also write $\Re(w_0)=\frac{1}{2}(w_0+\bar w_0)$.] One has $\mathop{\lim}\limits_{p=(z_0,w_0)\to 0=(0,0)}\pi_\star(z_0,w_0)=0$, and $\mathop{\lim}\limits_{p=(z_0,w_0)\to \infty=(0,2)}\pi_\star(z_0,w_0)=\frac{0}{0}$. On the other hand for continuity we can assume $\mathop{\lim}\limits_{p=(z_0,w_0)\to \infty=(0,2)}\pi_\star(z_0,w_0)=\infty\in \mathbb{C}$.
Therefore, $\pi_\star:\mathbb{C}\to\hat S^1$ is a derivable function with a simple zero at $0\in \hat S^1$ and a simple pole
$\infty\in\hat S^1$. Warn ! $\pi_\star$ cannot be holomorphic since any holomorphic map $\hat S^1\to\mathbb{C}$ is constant. This can be also
seen since $(\partial \bar z_0.\pi_\star)=-\frac{z_0}{(2-w_0)^2(\Re(w_0)-1)}$, hence $(\partial \bar z_0.\pi_\star)=0$ only at $z_0=0$ and
not on all $\hat S^1$. Moreover, the stereographic projection $\pi_\star$ is a proper mapping. In fact any continuous map from a
compact space to a Hausdorff space is proper (and closed). Therefore we can apply to $\pi_\star$ the Riemann-Hurwitz theorem
to identify its total branching index. (See footnote in Lemma \ref{riemann-hurwitz-formula}.) In fact we get
$[\chi(\hat S^1)=2]=[{\rm deg}(\pi_\star)=2]\cdot[\chi(\mathbb{C})=1]-b$, hence $b=0$.} There are also represented the first
two zeros on $\Gamma$, besides some other symmetric couple of zeros identified by zeros of $\tilde\zeta$ on the critical lines.
In Fig.~
\ref{appendix-a-completed-critical-line-on-quantum-critical-line} is given a different representation of the curve $\hat\zeta(t)$ on the Riemann sphere $\hat S^1$.
\end{aremark}
\begin{figure}[h]
 \includegraphics[height=3.5cm]{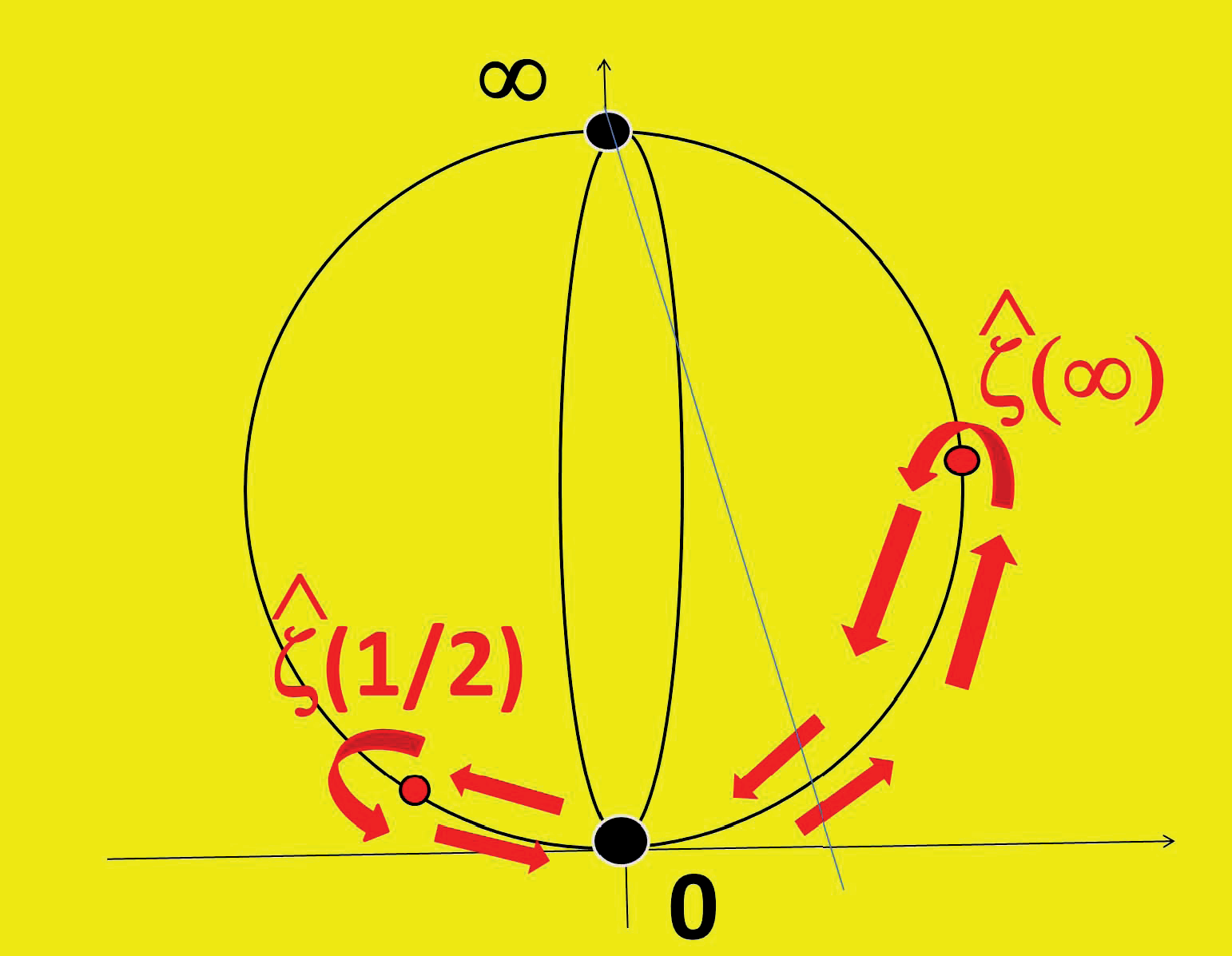}
\renewcommand{\figurename}{Fig.}
\caption{Representation of the quantum-complex curve $\hat\zeta(t)\subset\hat S^1$. In the picture $\hat\zeta(t)$ is drawn as a circulation of arrows. The points $\hat\zeta(\frac{1}{2})$ and $\hat\zeta(\infty)$ are branching points of $\hat\zeta$.}
\label{appendix-a-completed-critical-line-on-quantum-critical-line}
\end{figure}

\appendix{\bf Appendix B: Proof that $\mathbf{\hat b:\mathbb{C}\to\hat S^1}$ is not a holomorphic function .}\label{appendixb}
\renewcommand{\theequation}{B.\arabic{equation}}
\setcounter{equation}{0}  % reset counter
In this appendix we shall prove that the mapping $\hat b:\mathbb{C}\to\hat S^1$ considered in the commutative diagram (\ref{commutative-diagram-quantum-zeta-riemann-function-a}) is not holomorphic. For this it is enough to prove that the mapping $\phi$ is not holomorphic, where $\phi:\mathbb{C}\to\mathbb{C}$ is the mapping entering in the composition $\hat b=\hat a\circ \phi$. In fact $\phi(s)=\phi(x+i y)=u(x,y)+iv(x,y)$ with $u(x,y)=x$ and $v(x,y)=z_k+\frac{y-z_k}{z_{k+1}-z_k}$, if $z_k\le y\le z_{k+1}$. Thus $u$ and $v$ have first partial derivative, even if these are not continuous everywhere. But this is not sufficient to conclude that $\phi$ is not holomorphic. In fact, we can use the following lemma.

\begin{blemma}[Looman-Menehoff theorem]\label{appendix-b-looman-menehoff-theorem}
If $f=u+iv$ is continuous and $u$ and $v$ have everywhere but a countable set partial derivative, and they satisfy the Cauchy-Riemann equations {\em(\ref{appendix-b-cauchy-riemann-equations})},
\begin{equation}\label{appendix-b-cauchy-riemann-equations}
    \{(\partial x.u)=(\partial y.v),\,  (\partial y.u)=-(\partial x.v)\}\, \Leftrightarrow\, \{(\partial \bar z.f)=\frac{1}{2}((\partial x.f)+(\partial y.f))=0\},
\end{equation}

then $f$ is holomorphic.
\end{blemma}
\begin{proof}
See, e.g., \cite{LOOMAN, MENCHOFF, MONTEL}.
\end{proof}
Therefore, in order to state that $\phi$ is not holomorphic it is enough to verify that it does not verify equations (\ref{appendix-b-cauchy-riemann-equations}). In fact $(\partial x.u)=1\not=(\partial y.v)=\frac{1}{z_{k+1}-z_k}$, even if $(\partial y.u)=0=(\partial x.v)$. This concludes the proof that $\hat b$ is not a holomorphic mapping.

\appendix{\bf Appendix C: Proof that $\mathbf{\tilde\zeta\circ\pi_*:\hat S^1\to\hat S^1}$ is not a holomorphic function .}\label{appendixc}
\renewcommand{\theequation}{C.\arabic{equation}}
\setcounter{equation}{0}  % reset counter

Let us recall that $\pi_\star:\hat S^1\to\mathbb{C}$ is proper closed map and $\tilde\zeta:\mathbb{C}\to\hat S^1$ is meromorphic. But these properties do not assure that their composition $\hat \zeta_{trivial}=\tilde\zeta\circ\pi_*:\hat S^1\to\hat S^1$ is a meromorphic mapping, hence neither a quantum-complex zeta Riemann function. In fact the first requirement in such a case should be that $\hat \zeta_{trivial}$ is holomorphic. In the following we will prove that this cannot be the case. Let us first underline that $\hat \zeta_{trivial}$ is not constant. In fact $\hat \zeta_{trivial}$ sends $\hat 0$ and $\hat 1$ onto $\infty$, but it sends to $0$ the points $\hat z_k$ on the characteristic circle $\Gamma_{\frac{1}{2}}$, corresponding to the zeros $z_k$ of $\tilde\zeta$ on the characteristic line. Now if $\hat \zeta_{trivial}$ should be holomorphic it should be an open mapping. This property has the following consequence.

($\bigstar$)\hfill The set of points $r\in\hat S^1$ with $\hat \zeta_{trivial}(r)=s$, for fixed $s\in\hat S^1$, has no accumulation point in $\hat S^1$.

Let us now use the following lemmas.
\begin{clemma}[Topological properties of $\mathbb{Z}\subset\mathbb{R}$]\label{appendix-c-c1lemma}
Let us consider $\mathbb{Z}$ as a subset of $\mathbb{R}$. Then one has the following topological properties for $\mathbb{Z}$.

{\rm(i)} The interior of $\mathbb{Z}$ is empty: ${\rm int}(\mathbb{Z})=\varnothing$.

{\rm(ii)} The closure of $\mathbb{Z}$ coincides with $\mathbb{Z}$: $\overline{\mathbb{Z}}=\mathbb{Z}$ and the boundary coincides with $\mathbb{Z}$, $\partial\mathbb{Z}=\mathbb{Z}$.

{\rm(iii)} The derived set $\mathbb{Z}'$ of $\mathbb{Z}$, i.e., the set of accumulation points of $\mathbb{Z}\subset\mathbb{R}$ is empty: $\mathbb{Z}'=\varnothing$.

{\rm(iv)} $\mathbb{Z}$ is a closed subset of $\mathbb{R}$.

{\rm(v)} $\mathbb{Z}$ is a complete topological space.
\end{clemma}
\begin{proof}
(i) A point $n\in\mathbb{Z}$ is an interior point iff there exists an open interval $]\alpha,\beta[\subset\mathbb{R}$ containing $n$ such that it is contained into $\mathbb{Z}$. But this is impossible !

(ii) Since one has $\overline{\mathbb{Z}}={\rm int}(\mathbb{Z})\bigcup\partial\mathbb{Z}$, it follows from (i) that $\overline{\mathbb{Z}}=\partial\mathbb{Z}$. On the other hand, $\partial\mathbb{Z}=\{p\in\mathbb{R}\, |\, p\not\in{\rm int}(\mathbb{Z}),\, p\not\in{\rm ext}(\mathbb{Z})\}$, Since ${\rm ext}(\mathbb{Z})={\rm int}(\complement(\mathbb{Z}))$, hence $\partial\mathbb{Z}=\mathbb{Z}$.

(iii) $p\in\mathbb{Z}'$ iff any open set $U\subset\mathbb{R}$ containing $p\in\mathbb{R}$, contains also a point $a\in\mathbb{Z}$, with $a\not=p$. Then for any open interval $]\alpha,\beta[\subset\mathbb{R}$ containing $p$ one should have some integer $n\in\mathbb{Z}$, $n\not= p$, with $n\in\in]\alpha,\beta[$. But this is impossible, since for any $p\in\mathbb{R}$ let $n(p)$ the nearest integer to $p$. Then $]\alpha,\beta[$ can be taken centered on $p$ and with $n(p)<\alpha<p$. This is surely an open set containing $p$, but it cannot contain any integer. Since this holds for any $p$, it follows that the set $\mathbb{Z}'=\varnothing$.

(iv) $\mathbb{Z}$ is a closed subset of $\mathbb{R}$ since $\complement\mathbb{Z}=\bigcup_{n}]n,n+1[$, i.e., the complement of $\mathbb{Z}$ is the union of open intervals of $\mathbb{R}$, hence it is an open subset of $\mathbb{R}$.

(v) $\mathbb{Z}$ is a complete topological space, since all its Cauchy sequences are of the type $<a_1.a_2,\cdots,a_k,b,b.b\cdots>$ with $a_i,b\in\mathbb{Z}$, therefore have as convergence point $b\in\mathbb{Z}$.\footnote{Note that this is an example where one can understand the difference between the concept of accumulation point and the one of point of convergence of a Cauchy sequence. Compare properties (iii) with (v) in above Lemma C\hskip-0.5pt\ref{appendix-c-c1lemma}.}
\end{proof}
\begin{clemma}[Topological properties of $\mathbb{Z}\subset \mathbb{R}^{+}=\mathbb{R}\bigcup\{\infty\}\backsimeq S^1$]\label{appendix-c-c2lemma}
The topological properties of $\mathbb{Z}$, as subset of the Alexandrov compactified of $\mathbb{R}$, i.e., $S^1$, are the same of the ones when one considers $\mathbb{Z}\subset\mathbb{R}$, except with respect to the property {\em(iii)} in Lemma C\hskip-0.5pt\ref{appendix-c-c1lemma} that must instead be substituted with $\mathbb{Z}'=\{\infty\}$.
\end{clemma}
\begin{proof}
The proof can be conduced directly by similarity with the one of Lemma C\hskip-0.5pt\ref{appendix-c-c1lemma}.
\end{proof}
By using Lemma C\hskip-0.5pt\ref{appendix-c-c1lemma} and Lemma C\hskip-0.5pt\ref{appendix-c-c2lemma} we get that $\hat \zeta_{trivial}$ has $\infty$ as accumulation point of its set of zeros on $\Gamma_{\frac{1}{2}}$. In fact this set has the same cardinality $\aleph_0$ of $\mathbb{Z}$. Therefore, by using the property ($\bigstar$), we can conclude that $\hat \zeta_{trivial}$ cannot be a holomorphinc mapping $\hat S^1\to\hat S^1$, hence neither a meromorphic one.\footnote{Let us emphasize that in our proof it does not necessitate assume that $\tilde\zeta$ has only zeros on the critical line. In fact, whether $\tilde\zeta$ should have also some zeros outside the critical line, the set $\{\hat\zeta_{trivial}^{-1}(0)\}$ should yet admit $\infty$ as accumulation point. So whether $\hat \zeta_{trivial}$ should be holomorphic it should be constant in a neighborhood of $\infty$, and hence, everywhere. But $\hat \zeta_{trivial}$ is not constant. Therefore, also admitting zeros outside the critical lines, $\hat \zeta_{trivial}$ cannot be holomorphic or meromorphic.}

\appendix{\bf Appendix D: Complementary subjects supporting the proof of Theorem \ref{main-proof}.}\label{appendixd}
\renewcommand{\theequation}{D.\arabic{equation}}
\setcounter{equation}{0}  % reset counter

In this appendix we shall recall some complementary information supporting some statements used in the proof of Theorem \ref{main-proof} (main-proof). Then for such information we refer to some well-known early results about the Riemann hypothesis. Furthermore, here we shall also show that the existence of the homotopic morphism between $\tilde\zeta$ and $\hat\zeta$ works well also when the cardinality of the outside-ghost zeros set of $\tilde\zeta$ is at the beginning assumed to be $\aleph_0$.

$\bullet$\hskip 2pt {\bf Zero-free region - }In the main-proof we have assumed that whether some nontrivial zeros of $\zeta(s)$ (and hence of $\tilde\zeta(s)$) should exist critical zeros in the critical strip $\Xi$ ($0<\Re(s)<1$) they should belong to a sub-strip $\Xi_0$ strictly contained into $\Xi$. In other words non-trivial zeros cannot be too close to the boundary of $\Xi$, identified by the lines $s=0$ and $s=1$. This is a well-known result first obtained by De la Vall\'ee-Poussin \cite{DELAVALLE-POUSSIN} that in particular proved that if $\sigma+it$ is a zero of $\zeta(s)$, then $1-\sigma\ge \frac{C}{\log(t)}$ for some positive constant $C$. This zero-free region has been enlarged by several authors. For more details about see, e.g.,
 $\href{http://en.wikipedia.org/wiki/Riemann_hypothesis}{http://en.wikipedia.org/wiki/Riemann_hypothesis}$.

$\bullet$\hskip 2pt {\bf Numerical calculations - } Numerical calculations made with $\tilde\zeta(s)$ show that the first zeros that are present in critical strip $\Xi$ are present only on the critical line $\Xi_c$. Therefore one can without problem assume that the first two zeros, nearest to the $x$-axis are necessarily on the critical line $s=\frac{1}{2}\in\mathbb{C}$. (For details about see, e.g., $\href{http://en.wikipedia.org/wiki/Riemann_hypothesis}{http://en.wikipedia.org/wiki/Riemann_hypothesis}$.)

$\bullet$\hskip 2pt {\bf Infinite outside-ghost zeros assumption - } In the main-proof we have assumed that the number of zeros of $\tilde\zeta(s)$, outside the critical line is finite. However, the proof works well also by assuming that such a set has cardinality $\aleph_0$. In fact, if we accept this guess we can yet prove that $\tilde\zeta$ and $\hat\zeta$ are related by a homotopic morphism similarly to what made at page \pageref{divisor-homotopy}. Really let us assume that $\tilde\zeta$ has the divisor reported in (\ref{appendix-d-infinite-divisor-guess}).
\begin{equation}\label{appendix-d-infinite-divisor-guess}
    (\tilde\zeta)=\sum_{k\ge 1}1\cdot q_{s_k}+\sum_{j\ge 0}1\cdot q_{\pm z_j}-1\cdot q_0-1\cdot q_1
\end{equation}
where $q_{s_k}$ are the outside-ghost zeros and $q_{\pm z_j}$ are the critical-zeros of $\tilde\zeta$. Then let us consider the continuous deformation of the strip $\Xi_0$, where are contained all the zeros, onto the critical line $\Xi_c$. This deformation is characterized by a continuous flow $\phi_t:\mathbb{C}\to\mathbb{C}$, $(\phi^1_t,\phi^2_t)=(x\circ\phi_t,y\circ\phi_t)$, given in (\ref{appendix-d-continuous-flow-deforming-zeros-strip}).
\begin{equation}\label{appendix-d-continuous-flow-deforming-zeros-strip}
  \phi_t(x,y)=\left\{
  \begin{array}{ll}
    \phi^1_t(x,y) &=\left\{ \begin{array}{ll}
                              x&,\, 1-a<x<a  \\
                              x+t(\frac{1}{2}-x)&,\,  a\le x\le \frac{1}{2} \\
                              x-t(x-\frac{1}{2})&,\, \frac{1}{2}\le x\le 1-a \\
                            \end{array}\right\}\\
    \phi^2_t(x,y)& =y
  \end{array}
  \right.
\end{equation}
where $x=a$ and $x=1-a$ are the equations of the boundaries of the strip $\Xi_0$. The corresponding velocity $\dot\phi(x,y)=\dot\phi^1(x,y)\partial x+\dot\phi^2(x,y)\partial y$, has the components given in (\ref{appendix-d-velocity-continuous-flow-deforming-zeros-strip}).
\begin{equation}\label{appendix-d-velocity-continuous-flow-deforming-zeros-strip}
     \dot\phi(x,y)=\left\{
  \begin{array}{ll}
    \dot\phi^1(x,y) &=\left\{ \begin{array}{ll}
                              0&,\, 1-a<x<a  \\
                              \frac{1}{2}-x&,\,  a\le x\le \frac{1}{2} \\
                              \frac{1}{2}-x&,\, \frac{1}{2}\le x\le 1-a \\
                            \end{array}\right\}\\
    \dot\phi^2(x,y)& =0.
  \end{array}
  \right.
\end{equation}
Therefore the continuous flow $\phi_t$, $t\in[0,1]$, is singular, with the following set of singular points: $\{(a,y),(\frac{1}{2},y),(1-a,y)\}$, i.e., the right lines parallel to the $y$-axis, coinciding with the boundaries of $\Xi_0$ and the critical line $\Xi_c$ respectively.

We get the exact and commutative diagram (\ref{appendix-d-exact-commutative-diagram-infinite-divisor-guess})

\begin{equation}\label{appendix-d-exact-commutative-diagram-infinite-divisor-guess}
\xymatrix@C=1cm{
  \mathbb{C}\ar[d]_{\tilde \zeta}\ar[r]^{\phi_t}&
  \mathbb{C}\ar[d]^{\tilde\zeta_t}\ar[r]&0\\
  \hat S^1\ar[d]\ar[r]_{\phi'_t}&\hat S^1\ar[d]\ar[r]&0\\
  0&0&\\}
\end{equation}
where $\tilde\zeta_t$ is the meromorphic function on $\mathbb{C}$ associated to the divisor $D_t$ reported in (\ref{appendix-d-deformed-infinite-divisor-guess}).
\begin{equation}\label{appendix-d-deformed-infinite-divisor-guess}
    D_t=\sum_{k\ge 1}1\cdot q_{s_{k,t}}+\sum_{j\ge 0}1\cdot q_{\pm z_j}-1\cdot q_0-1\cdot q_1
\end{equation}
where $q_{s_{s,t}}=\phi_t( q_{s_k})$. The existence of the meromorphic functions $\tilde\zeta_t$ with divisor $D_t$ is assured by the Weierstrass' theorem. (See, e.g., \cite{HAZEWINKEL}.) Then $(\tilde\zeta_0)=(\tilde\zeta)$ and $(\tilde\zeta_1)=\sum_{\bar k\ge 1}1\cdot q_{s_{\bar k}}+\sum_{j\ge 0}1\cdot q_{\pm z_j}-1\cdot q_0-1\cdot q_1$, where $q_{s_{\bar k}}=\phi_1(q_{s_{k}})$. Then we can apply to $\tilde\zeta_1$ the morphism $(\hat b,\hat b')$ in such a way to obtain the exact and commutative diagram (\ref{appendix-d-exact-commutative-diagram-infinite-divisor-guess-a})
\begin{equation}\label{appendix-d-exact-commutative-diagram-infinite-divisor-guess-a}
\xymatrix@C=1cm{\mathbb{C}\ar[d]_{\tilde \zeta_1}\ar[r]^{\underline{\hat b}_1}&
  \hat S^1\ar[d]^{\hat\zeta}\ar[r]&0\\
  \hat S^1\ar[d]\ar[r]_{\underline{\hat b}'_1}&\hat S^1\ar[d]\ar[r]&0\\
  0&0&\\}
  \end{equation}

that works since $\tilde\zeta_1$ has all zeros on the critical line. Then we get the homotopy morphism relating $\tilde\zeta$ and $\hat\zeta$ given in (\ref{appendix-d-exact-commutative-diagram-infinite-divisor-guess-b}).
\begin{equation}\label{appendix-d-exact-commutative-diagram-infinite-divisor-guess-b}
\hat\beta_t=(\hat b_t,\hat b'_t)=\left\{
\begin{array}{ll}
  (\phi_t,\phi'_t)& t\in[0,1) \\
  (\underline{\hat b}_t\circ\phi_t,\underline{\hat b}'_t\circ\phi'_t)& t=1.\\
\end{array}
\right.
\end{equation}

Therefore, similarly to the diagram (\ref{exact-commutative-diagram-homotopy-zeros}) we get the commutative diagram (\ref{appendix-d-exact-commutative-diagram-infinite-divisor-guess-c}) that realizes the homotopic morphism between $\tilde\zeta$ and $\hat\zeta$.

 \begin{equation}\label{appendix-d-exact-commutative-diagram-infinite-divisor-guess-c}
\xymatrix{\mathbb{C}\ar@/^4pc/[rrrr]^{\hat b_1}\ar@/^3pc/[rrr]_(0.8){\phi_1}\ar[d]_{\tilde \zeta}\ar[r]_{\hat b_t=\phi_t}&\mathbb{C}\ar[d]^{\tilde\zeta_t}\ar[r]&\cdots\ar[r]&\mathbb{C}\ar[d]_{\tilde\zeta_1}\ar[r]_{\underline{\hat b}_1}&
  \hat S^1\ar[d]^{\hat\zeta}\ar[r]&0\\
  \hat S^1\ar@/_5pc/[rrrr]_{\hat b'_1}\ar@/_3pc/[rrr]^(0.8){\phi'_1}\ar[d]\ar[r]_{\hat b'_t=\phi'_t}&\hat S^1\ar[d]\ar[r]&\cdots\ar[r]&\hat S^1\ar[r]^{\underline{\hat b}'_1}\ar[d]&\hat S^1\ar[d]\ar[r]&0\\
  0&0&&0&0\\}
  \end{equation}

  Therefore also under the guess that $\tilde\zeta$ can have infinite outside-ghost zeros, we are able to prove that $\hat\Omega_\bullet(\tilde\zeta)\cong\hat\Omega_\bullet(\hat\zeta)$. Then the same conclusion made for the case of finite outside-ghost zeros follows.

\appendix{\bf Appendix E: Proof of the isomorphism $\mathbf{\hat\Omega_\bullet(\hat\zeta)\cong\hat\Omega_\bullet(\tilde\zeta)}$.}\label{appendixe}
\renewcommand{\theequation}{E.\arabic{equation}}
\setcounter{equation}{0}  % reset counter

In this appendix we give a detailed proof that the homomorphism $\hat\Omega_\bullet(\hat\beta_1):\hat\Omega_\bullet(\tilde\zeta)\to\hat\Omega_\bullet(\hat\zeta)$, induced by the homotopy morphism $\hat\beta_t$, for $t=1$, as considered in the diagrams (\ref{exact-commutative-diagram-homotopy-zeros}) and (\ref{short-exact-sequences}), is an isomorphism.  Let us first prove that $\hat\Omega_\bullet(\hat\beta_1)$ is a monomorphism. For this it is enough to prove that if two morphisms $\alpha_i=(a_i.a'_i):\hat f\to \tilde\zeta$, $i=1,2$, with $a_i$, $a'_i$, closed  quantum-complex maps, are not equivalent, the induced morphisms

\begin{equation}\label{appendix-e-intermediate-calculation-a}
\beta_i=(\hat d_i=\hat b_i\circ a_i,\hat d'_i=\hat b'_i\circ a'_i):\hat f_i\to\hat\zeta,\, \left\{
\begin{array}{l}
  a_i:M_i\to \hat S^1 \\
  a'_i:N_i\to \hat S^1 \\
  \hat f_i:M_i\to N_i\\
\end{array}
\right\}
\end{equation}

cannot be equivalent too, with respect to $\hat\zeta$. In fact, if they should be equivalent, with respect to $\hat\zeta$, it should exist a morphism

\begin{equation}\label{appendix-e-intermediate-calculation-b}
\Phi=(c,c'):(\hat\phi,\hat V,\hat W)\to\hat\zeta,\, \left\{
\begin{array}{l}
  \partial \hat V=M_1\sqcup M_2 \\
\partial \hat W=N_1\sqcup N_2 \\
\hat\phi|_{\partial\hat V} =\hat f_1\sqcup \hat f_2:M_1\sqcup M_2\to N_1\sqcup N_2
\end{array}
\right\}.
\end{equation}

On the other hand, taking into account the fibration structures $\hat b_1:\mathbb{C}\to \hat S^1$ and $\hat b'_1:\hat S^1\to \hat S^1$ one can build the commutative diagram (\ref{appendix-e-commuatative-diagram-monomorphism}).

\begin{equation}\label{appendix-e-commuatative-diagram-monomorphism}
   \xymatrix{&\hat S^1\ar[r]^{\hat\zeta}&\hat S^1&\\
   M_i\ar@/_1pc/[rrr]_{\hat f_i}\ar[ur]^{\hat d_i=\hat b_i\circ a_i}\ar[r]^{a_i}\ar@{^{(}->}[d]& \mathbb{C}\ar[u]^{\hat b_1}\ar[r]^{\tilde\zeta}&\hat S^1\ar[u]_{\hat b'_1}&N_i\ar[l]_{a'_i}\ar@{^{(}->}[d]\ar[ul]_{\hat d'_i=\hat b'_i\circ a'_i}\\
   \hat V\ar@{=}[d]\ar@/^5pc/[uur]^{c}\ar[ur]^{l}\ar[rrr]_{\hat\phi}&&&\hat W\ar@{=}[d]\ar[ul]_{l'}\ar@/_5pc/[uul]_{c'}\\
   \frame{\includegraphics[height=0.8cm]{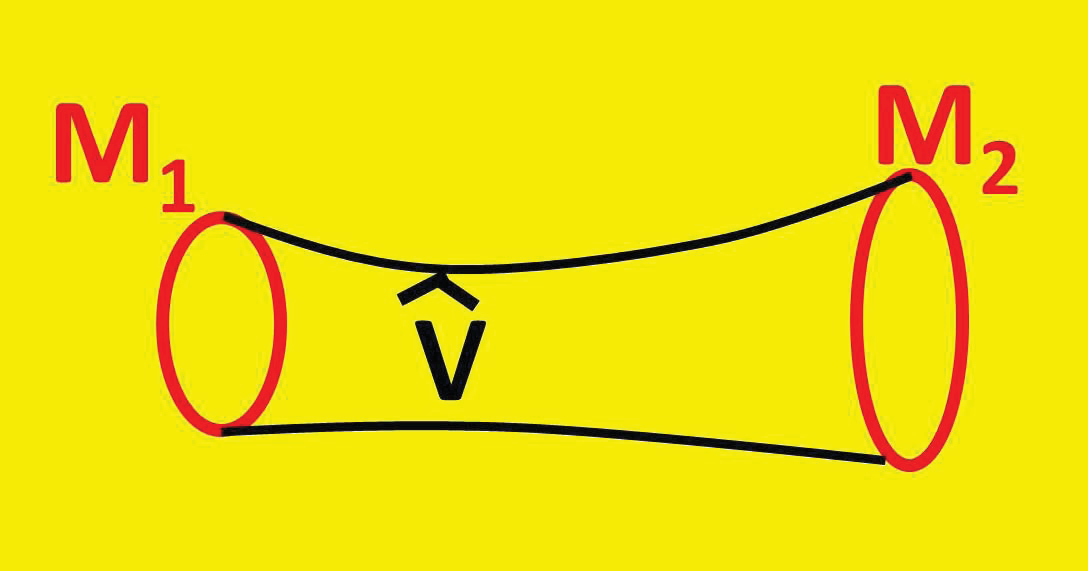}}\ar[rrr]_{\hat\phi}&&&\frame{\includegraphics[height=0.8cm]{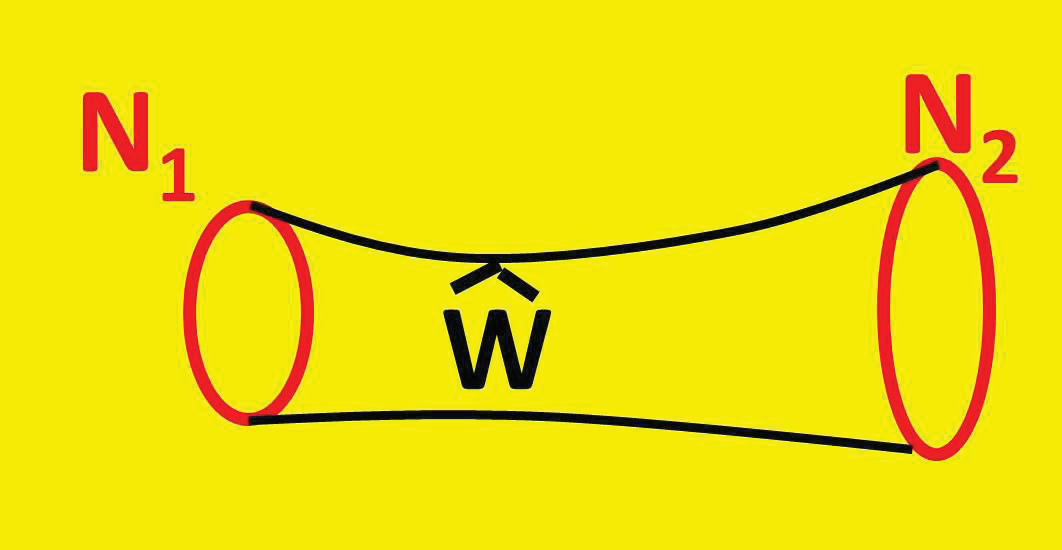}}}
\end{equation}
More precisely, taking into account the fibration structures $\hat b_1:\mathbb{C}\to \hat S^1$ and $\hat b'_1:\hat S^1\to \hat S^1$, we can (non-canonically) identify liftings $l:\hat V\to \mathbb{C}$, $l':\hat W\to \hat S^1$ such that $c=\hat b_1\circ l$ and $c'=\hat b'_1\circ l'$. Then since $c'\circ\hat\phi=\hat\zeta\circ c$, we get
\begin{equation}\label{appendix-e-intermediate-calculation-c}
    \left\{\begin{array}{ll}
      \hat b'_1\circ l'\circ\hat\phi &= \hat\zeta\circ\hat b_1 \circ l \\
      &= \hat b'_1\circ\tilde\zeta\circ l \\
    \end{array}\right\} \, \Rightarrow \, l'\circ\hat\phi=\tilde\zeta\circ l.
\end{equation}

In other words $\hat f_i$, $i=1.2$, should be equivalent also with respect to $\tilde\zeta$, despite the assumption that they are not so. Therefore, the homomorphism $\hat\Omega_\bullet(\hat\beta_1)$ is a monomorphism.

To prove that $\hat\Omega_\bullet(\hat\beta_1)$ is an epimorphism it is enough to use again the property of the fibration structures $\hat b_1:\mathbb{C}\to \hat S^1$ and $\hat b'_1:\hat S^1\to \hat S^1$. In fact, for any equivalence class of $\hat\Omega_\bullet(\hat\zeta)$, encoded by a morphism $\alpha=(a,a'):\hat f\to\hat\zeta$, we can identify liftings of $h$ and $h'$ respectively such that $\hat b_1\circ h=a$, $\hat b'_1\circ h'=a'$. (See commutative diagram (\ref{appendix-e-commuatative-diagram-monomorphism-a}).) Then one can see, by utilizing some steps similar to previous ones, that $\gamma=(h,h'):\hat f\to\tilde\zeta$. In other words $\hat f$ identifies also an equivalence class of $\hat\Omega_\bullet(\tilde\zeta)$. By conclusion since $\hat\Omega_\bullet(\hat\beta_1)$ is an epimorphism and a monomorphism, it follows that $\hat\Omega_\bullet(\hat\beta_1)$ is an isomorphism.

\begin{equation}\label{appendix-e-commuatative-diagram-monomorphism-a}
\xymatrix{&0&0&\\
&\hat S^1\ar[r]^{\hat\zeta}\ar[u]&\hat S^1\ar[u]&\\
&\mathbb{C}\ar[u]^{\hat b_1}\ar[r]^{\tilde\zeta}&\hat S^1\ar[u]^{\hat b'_1}&\\
M\ar@/^2pc/[uur]^{a}\ar[ur]_{h}\ar[rrr]_{\hat f}&&&N\ar@/_2pc/[uul]^{a'}\ar[ul]^{h'}\\}
\end{equation}
To conclude this appendix let us note that the surjective morphism $\hat\beta_1:\tilde\zeta\to\hat\zeta$ exists also in the case when one assumes that $\tilde\zeta$ has infinite zeros outside the critical line. (See Appendix D.) Therefore above arguments apply also in this case to prove the isomorphism $\hat\Omega_\bullet(\tilde\zeta)\cong \hat\Omega_\bullet(\hat\zeta)$. Finally, whether $\tilde\zeta$ has not zeros outside the critical line, then $\hat\beta_1$ can be substituted in above calculations with $\hat\beta$.
\end{appendices}

\end{document}